\documentclass{article}

\usepackage[english]{babel}

\usepackage[letterpaper,top=2cm,bottom=2cm,left=3cm,right=3cm,marginparwidth=1.75cm]{geometry}

\usepackage{amsmath, amsthm, amssymb}
\usepackage{graphicx}
\usepackage[colorlinks=true, allcolors=blue]{hyperref}

\usepackage{natbib}
\usepackage{thmtools}
\declaretheorem[name=Theorem,numbered=no]{reptheorem}
\usepackage{authblk}

\renewcommand{\Pr}{\field{P}}

\newcommand{\ba}{\boldsymbol{a}}

\newcommand{\bg}{\boldsymbol{g}}

\newcommand{\bx}{\boldsymbol{x}}

\newcommand{\by}{\boldsymbol{y}}

\newcommand{\bxi}{\boldsymbol{\xi}}

\newcommand{\bzero}{\boldsymbol{0}}

\newcommand{\btheta}{\boldsymbol{\theta}}

\newcommand{\field}[1]{\mathbb{#1}}

\newcommand{\R}{\field{R}}

\newcommand{\E}{\field{E}}

\newcommand\inn[2]{ \left\langle {#1} \,,\, {#2} \right\rangle }

\newcommand{\norm}[1]{\left\|{#1}\right\|}

\newcommand{\F}{\mathcal{F}}

\newcommand{\hn}{\widehat{\nabla}f}
\newtheorem{theorem}{Theorem}[section]
\newtheorem{corollary}{Corollary}[theorem]
\newtheorem{lemma}[theorem]{Lemma}
\newtheorem{remark}[theorem]{Remark}

\newtheorem{assumption}{Assumption}
\newtheorem{proposition}{Propositon}

\title{Dual Averaging Converges for Nonconvex Smooth Stochastic Optimization}

\author[1]{Tuo Liu\thanks{\texttt{tuo.liu@kaust.edu.sa}}}
\author[1]{El Mehdi Saad\thanks{\texttt{mehdi.saad@kaust.edu.sa}}}
\author[2]{Wojciech Kotłowski\thanks{\texttt{wkotlowski@cs.put.poznan.pl}}}
\author[1]{Francesco Orabona\thanks{\texttt{francesco@orabona.com}}}

\affil[1]{King Abdullah University of Science and Technology, Saudi Arabia}
\affil[2]{Poznan University of Technology, Poland}

\begin{document}

\maketitle

\begin{abstract}
Dual averaging and gradient descent with their stochastic variants stand as the two canonical recipe books for first-order optimization: Every modern variant can be viewed as a descendant of one or the other.
In the convex regime, these algorithms have been deeply studied, and we know that they are essentially equivalent in terms of theoretical guarantees. 
On the other hand, in the non-convex setting, the situation is drastically different: While we know that SGD can minimize the gradient of non-convex smooth functions, no finite-time complexity guarantee for Stochastic Dual Averaging (SDA) was known in the same setting.
In this paper, we close this gap by a reduction that views SDA as SGD applied to a sequence of implicitly regularized objectives.
We show that a tuned SDA exhibits a rate of convergence $\mathcal{O}(1 / T + \sigma \log T/ \sqrt{T})$, similar to that of SGD under the same assumptions.
To our best knowledge, this is the first complete convergence theory for dual averaging on non-convex smooth stochastic problems without restrictive assumptions, closing a long-standing open problem in the field.
Beyond the base algorithm, we also discuss ADA-DA, a variant that marries SDA with AdaGrad's auto-scaling, which achieves the same rate without requiring knowledge of the noise variance. 
\end{abstract}

\section{Introduction}
\label{sec:intro}

Stochastic first-order methods have become the computational workhorse for large-scale optimization problems.  
Among these methods, \emph{Dual Averaging} (DA) offers a principled alternative to classical Stochastic Gradient Descent (SGD) by updating the primal iterate through a running average of past gradients.
Although a decade of work has thoroughly characterized DA in convex settings, its behaviour on objectives that are smooth and possibly non-convex remains poorly understood, leaving open the question of whether DA can attain the $\mathcal{O}(T^{-1/2})$ complexity to minimize the expected squared norm of the gradients enjoyed by well-tuned SGD in the same setting.

One might wonder why studying DA in the unconstrained setting is worthwhile. Indeed, a very common misunderstanding is that DA and gradient descent only differ in the constrained setting. In fact, the two methods already diverge in the unconstrained setting, because they incorporate new sub-gradients in fundamentally different ways. The origins of DA trace back to Nesterov's seminal work~\citep{Nesterov09}, where the author observed that the stepsize conditions required for subgradient descent,
\[
\eta_t > 0, \quad \eta_t \to 0, \quad \sum_{t=1}^{\infty} \eta_t = \infty,
\]
necessarily make ``New subgradients enter the model with decreasing weights''~\citet[page 4]{Nesterov09}, contradicting the principle of recursive methods where fresher information should outweigh stale data. More formally, when initialized at $\bx_1 = \mathbf{0}$, gradient descent and DA generate the iterates\footnote{\citet{Nesterov09} also proposed the possibility in DA to weigh each single gradient by an additional parameter $\lambda_i$, in addition to $\eta_t$. While potentially useful in the deterministic setting, this idea has been abandoned in the stochastic setting~\citep[see, e.g., ][]{Xiao10}, hence we do not consider it here.}
\[
\bx^{\mathrm{GD}}_{t+1} = -\sum_{i=1}^{t} \eta_i \, \bg_i, \qquad \bx^{\mathrm{DA}}_{t+1} = -\eta_t \sum_{i=1}^{t} \bg_i,
\]
where $\bg_i$ is a subgradient of $f$ at $\bx_i$.  
The effective weight assigned to the latest gradient is therefore \(\eta_t / \bigl(\sum_{i=1}^{t} \eta_i\bigr)\) 
for subgradient descent, but \(1/t\) for DA.  
Because the sequence $(\eta_t)$ is decreasing, the factor \(1/t\) is larger; as a result DA places comparatively greater emphasis on the latest information. This property might be particularly valuable in non-convex landscapes, where earlier gradient directions may oppose the current descent direction, even in low-noise regimes (e.g., with the Rosenbrock function). Empirical evidence in~\cite{JelassiD20} corroborates this intuition, showing that DA-based optimizers seem to match or outperform their SGD counterparts on deep learning training tasks.

Hence, the gap in the theoretical literature on DA and the empirical evidence raises the question of whether it is possible to show that DA can minimize the gradient of smooth, possibly non-convex, functions.

\paragraph{Our contribution.}
We provide the first iterate-level guarantees for DA on \emph{smooth, possibly non-convex} objectives.  
For deterministic step sizes of order $1/\sqrt{t}$, our bounds recover the optimal $\mathcal{O}(T^{-1/2})$ rate previously established for SGD~\citep{GhadimiL12}, under the strong-growth (affine noise) assumption. In addition, we derive a high-probability bound under the assumption of sub-Gaussian noise.
We then address the question of designing an adaptive variant. When neither the smoothness constant nor the noise level is known, and the learner employs AdaGrad-style rates~\citep{McMahanS10,DuchiHS11}, we prove convergence bounds that depend solely on the maximal iterate norm.  Whenever the iterates remain bounded, the standard optimal rates follow immediately.

\paragraph{Organization of the paper.}
The remainder of the paper is organized as follows. Section~\ref{sec:related} surveys related work on dual averaging and its variants. Section~\ref{sec:prob-setup} formalizes the optimization framework and introduces our notation. Section~\ref{sec:strong-growth} establishes convergence rates for DA with deterministic step sizes of order~$1/\sqrt{t}$. In Section \ref{sec:high-prob}, we provide high-probability bounds. Finally, in Section~\ref{sec:ada-DA}, we extend the analysis to the parameter-free setting, deriving bounds for AdaGrad-style adaptive step sizes. 

\section{Related work}\label{sec:related}

Dual Averaging (DA) was introduced by \citet{Nesterov09} for the deterministic case. The extension of DA to the stochastic case was done in \citet{Xiao10}, focusing on the case of composite losses.
These algorithms are also known in the online convex optimization setting~\citep{Gordon99b,Zinkevich03,Orabona19}, a generalization of the stochastic setting, as Follow-The-Regularized-Leader with linear losses~\citep{Gordon99b,ShalevS06,ShalevS06b,AbernethyHR08,HazanK08}.
All these analyses are for convex functions.

In the non-convex setting, the literature on dual averaging is very sparse.
Clearly, when the learning rate is constant, SGD and DA coincide in the unconstrained setting. In this specific case, the seminal paper of \citet{GhadimiL12} shows that SGD/DA with the optimal constant learning rate give a $\mathcal{O}(\frac{1}{T}+\frac{\sigma}{\sqrt{T}})$ convergence in expectation for the squared norm of the gradient of a random iterate. Under conditions of sub-Gaussianity of the noise, they also proposed a two-stage procedure to select an iterate with small gradient with high probability.
\citet{LiO19} proved that SGD on non-convex smooth functions adapts to the noise level when using a delayed version of the AdaGrad-norm stepsizes~\citep{StreeterM10}, achieving the same rate above without knowledge of $\sigma$. Later, \citet{WardWB19} proved a similar guarantee for the original AdaGrad-norm stepsizes, but under the stronger assumption of bounded gradients. To the best of our knowledge, no results of this form are known for the DA version of AdaGrad~\citep{McMahanS10,DuchiHS11}.

\citet{SuggalaN20} show that follow-the-perturbed-leader, a variant of the follow-the-regularized-leader where the regularization is achieved through noise, with a minimization oracle can be used for online learning with non-convex losses. They also show that non-randomized algorithms cannot achieve vanishing average regret in the same setting.
There are analyses for modified versions of DA, where the gradients are rescaled by a ``learning rate'' \citep[see, e.g., ][]{OrabonaP21}. However, these variations essentially imitate the behavior of SGD in having a decreasing weight for the stochastic gradients, defeating the core motivation of DA of having equal weights for all gradients in the iterates.

\cite{JelassiD20} provide empirical evidence that DA works well for the stochastic optimization of smooth non-convex functions. They also provide a convergence rate, but their bound requires that the norm of the iterates is bounded by some constant $R$, moreover it fails to recover fast rates $\mathcal{O}(1/T)$ in the low noise regime $\sigma=0$.

More recently, \citet{ChenH24} proposed the open problem that stochastic smooth optimization can be directly reduced to online learning. If true, it would imply that the standard regret analysis of DA would immediately give us a convergence guarantee for stochastic smooth optimization. However, as far as we know, this is still an open problem.

\section{Problem Set-up} 
\label{sec:prob-setup}

We consider the unconstrained optimization problem
\[
    \min_{\bx \in \R^{d}} f(\bx),
\]
where $f: \R^{d} \to \R$ is the objective function. Let $\hn(\bx)$ denote the noisy gradient at the point $\bx$. For the iterates of stochastic dual averaging (SDA), we set $\bg_t := \hn(\bx_t)$ and denote by $\bxi_t$ the corresponding noise. Our analysis focuses on the convergence of SDA:
\begin{equation} \label{eq:DA-def}
    \bx_{t+1} = - \eta_{t} \sum_{i=1}^{t} \bg_t, \quad \bxi_{t} := \bg_t - \nabla f(\bx_t) , \quad t \geq 1,
\end{equation}
where $\eta_{t}$ is a sequence of step sizes and $\hn(\bx_{t})$ is a gradient oracle with noise $\bxi_{t}$. When the step sizes are constant, we recover the SGD algorithm. If in addition $\bxi_{t} = \bzero$ almost surely, we recover to the classic dual averaging algorithm. The distance between solutions in $\R^{d}$ is measured by the $\ell_{2}$ norm. In this paper, depending on the specific theorem, we use a subset of the following assumptions.

\begin{assumption}[$L$-smoothness]
\label{ass:smoothness}
The function $f \colon \R^{d} \to \R$ is differentiable and there exists a constant $L > 0$ such that for all $\bx, \by \in \R^{d}$:
\begin{equation}
\label{eq: smoothness}
    f(\by) \leq f(\bx) + \inn{\nabla f(\bx)}{\by - \bx} + \frac{L}{2} \norm{\by - \bx}^{2}. \tag{A1}
\end{equation}
\end{assumption}

\begin{assumption}[Lower boundedness]
\label{ass:lower-boundedness}
$f$ admits a finite lower bound, i.e.,
\begin{equation} 
\label{eq:lower-boundedness}
    \inf_{\bx \in \R^{d}} \ f(\bx) := f^{*} > -\infty. \tag{A2}
\end{equation}
\end{assumption}

\begin{assumption}[Unbiased estimator]
\label{ass:unbiased-estimator}
We have access to a history of independent, unbiased gradient estimator $\hn(\bx)$ for any $\bx \in \R^{d}$,
\begin{equation} \label{eq:unbiased-estimator}
    \E\left[\hn(\bx) \mid \bx\right] = \nabla f(\bx). \tag{A3}
\end{equation}
\end{assumption}

In most classical analyses of stochastic first-order methods one typically assumes access to an oracle with uniformly bounded variance (BV) \citep{harold1997stochastic}, i.e., $\E\left[\norm{\nabla f(\bx) - \hn(\bx)}^{2}\right] \leq \sigma^{2}$ for every $\bx$, yet it is provably too restrictive, even for the basic mini-batch least-squares problem, where gradient noise grows with the iterate norm so no finite global $\sigma^2$ is admissible \citep{jain2018parallelizing,zhang2019stochastic}.
Hence, we use the following assumption.

\begin{assumption}[Strong growth condition (SGC) with additive noise]
\label{ass:strong-growth}
The oracle $\hn(\cdot)$ satisfies the
strong growth condition with parameter $(\rho, \sigma)$, i.e.,
\begin{equation}
\label{eq:strong-growth}
    \E\left[\norm{\hn(\bx) - \nabla f(\bx)}^{2}\right] \leq \rho \norm{\nabla f(\bx)}^{2} + \sigma^{2}, \quad \forall \bx \in \R^{d}. \tag{A4}
\end{equation}
\end{assumption}
This assumption on the noise is similar to the one used in \citet{schmidt2013fast,bottou2010large,bottou2018optimization}. It \emph{strictly generalizes} the standard BV assumption (for which only $\rho = 0$ is admitted). By allowing the stochastic error to scale with the signal, the admissible $\sigma^2$ can be chosen to reflect only the residual noise at the optimum; in practice this constant is often orders of magnitude smaller than the global bound required when $\rho = 0$, leading to larger stable stepsizes and tighter convergence guarantees. Subsequent work \citep{mishkin2020interpolation,vaswani2019fast,solodkin2024accelerated} showed that many over-parameterized (interpolating) models satisfy SGC, and leveraged the condition to match accelerated deterministic rates without variance-reduction techniques.

\begin{assumption}[Sub-Gaussian noise]
\label{ass:subgaussian-noise}
$\norm{\hn(\bx) - \nabla f(\bx)}$ is a $\sigma$-sub-Gaussian random variable. There are several equivalent definitions of sub-Gaussian random variables up to an absolute constant scaling (see, e.g., Proposition 2.5.2 in \citet{vershynin2018high}). For convenience, we adopt the following definition. 
\begin{equation} 
\label{eq:subgaussian-noise}
    \E\left[\exp\left(\lambda^{2} \norm{\hn(\bx) - \nabla f(\bx)}^{2}\right) \mid \bx\right] 
    \leq \exp\left(\lambda^{2} \sigma^{2}\right), \text{ for all } \lambda \text{ such that }\left|\lambda\right| \leq \frac{1}{\sigma}. \tag{A5}
\end{equation}
\end{assumption}

Our results depend on the standard Assumptions~\ref{eq: smoothness}-\ref{eq:strong-growth}. For high probability guarantees we use Assumption~\ref{eq:subgaussian-noise} instead of Assumption~\ref{eq:strong-growth}.

\section{Bounds in expectation under the strong growth condition}
\label{sec:strong-growth}

In this section, we study the convergence of SDA, under the $(\rho, \sigma)$-strong growth condition, when using decreasing step sizes of the form  
\begin{equation}
\label{eq:eta}
    \eta_t = \frac{1}{L(1+\rho) (1+\rho + \alpha \sqrt{t})}~,
\end{equation}
 where $\alpha = \min \left\lbrace \frac{\sigma}{L(1+\rho)}, 1 \right\rbrace$. We show in Theorem~\ref{thm:exp}, that this schedule leads to a convergence rate for the average squared gradient norm of
\[
    \mathcal{O}\left(\frac{1}{T} + \frac{\sigma \log T}{\sqrt{T}}\right),
\]
where $\mathcal{O}(\cdot)$ hides polynomial dependence on the problem parameters $L, \sigma$ and $\rho$. Thus recovering the standard performance guarantees of SGD in the smooth non-convex setting, up to a logarithmic factor\footnote{The logarithmic term is due to the time-varying learning rates and one would suffer a similar term in SGD with time-varying learning rates.} in $T$ in the slow rate term.
Moreover, in regimes where the noise level is small relative to the gradients norm, the procedure achieves the better rate of $\mathcal{O}(1/T)$, comparable to the fast rates observed in noiseless or low-variance scenarios. The core idea of the analysis lies in reinterpreting SDA updates as an instance of SGD over a sequence of functions $(f_t)_{t\le T}$ defined below.

\paragraph{Main Challenges.}  

Traditional analysis of dual averaging in the convex setting goes through without pain, for the existence of linear surrogates of the loss $\tilde{\ell}_{t}(\bx) := \langle \sum_{i=1}^{t} \bg_i, \bx\rangle + \psi_{t}(\bx) / \eta_{t}$, for some sequence of regularizers $\psi_{t}(\cdot)$. Bregman telescoping yields $\mathcal{O}(\sqrt{T})$ regret and hence $\mathcal{O}(1 / \sqrt{T})$ optimization error in expectation after averaging \citep{Xiao10}. In the non-convex smooth case, however, we immediately lose this property. In general, especially with decreasing step sizes, every new iterate $\bx_{t+1}$ depends on the entire history of past stochastic gradients with non-vanishing weights. So viewing stochastic dual averaging (SDA) as an explicit primal descent step \eqref{eq:DA-def}, we run into unavoidable coupling terms like $\E\left[\langle \sum_{i=1}^{t-1} \bg_i, \nabla f(\bx_{t})\rangle\right]$, which classical analysis techniques cannot handle since a proper \emph{descent lemma} is missing. Recent literature confirms that new techniques are needed. \citet{liu2023high} analyzed non-convex SDA by replacing the convergence guarantee with the surrogate $\norm{\bx_{t+1} - \bx_{t}}^{2}$ and obtain only an $\mathcal{O}(1 / t)$ decay. \citet{cutkosky2019anytime} proposed an online-to-batch framework that bypasses the coupling by inserting an auxiliary projection, but at the price of departing from original SDA update that we care about.

\paragraph{Dual Averaging as SGD on Regularized Functions.}

We use an equivalent formulation of the SDA, presented in \cite{JelassiD20}, showing that it is equivalent to running SGD on a modified sequence of functions.
\begin{proposition}[\citep{JelassiD20}]\label{prop:DA-reformulation}
Let $\bx_{1} = \bzero$ and $\{f_{t}\}_{t=0}^{T}$ be a sequence of functions such that $f_{t} \colon \R^{d} \rightarrow \R$ and 
\[
    f_{t}(\bx) := f(\bx) + \frac{\gamma_{t}}{2} \norm{\bx}^{2}_{2},
\]
where $\gamma_{t} := \frac{1}{\eta_{t}} - \frac{1}{\eta_{t-1}}, \quad \forall t \geq 1$.
Then, the SDA update \eqref{eq:DA-def} is equivalent to SGD with learning rate $\eta_{t}$ on the function $f_{t}(\bx)$ for all $t \geq 1$.
\end{proposition}
This result is an immediate consequence of the update of SDA:
\[
\bx_{t+1} = -\eta_t \left(\bg_t + \sum_{i=1}^{t-1} \bg_i\right) = \bx_t -\eta_t \left(\bg_t+\left(\frac{1}{\eta_{t-1}}-\frac{1}{\eta_t}\right)\bx_t\right)~.
\]

\begin{remark}
Despite the ``equivalence'' result, we want to stress that an important difference between SDA and SGD. In fact, stochastic dual averaging (SDA) and stochastic gradient descent (SGD)
are built on different weighting schemes. SDA aggregates all past gradients with a \emph{uniform} weight \(1/t\), whereas SGD assigns the most recent gradient the diminishing factor
\(\eta_t\bigl/\sum_{i=1}^{t}\eta_i\). This uniform averaging changes the iterate trajectory, so SDA must be treated as an algorithm in its own right, rather than as a constrained-only variant of SGD. From a technical point of view, the standard smooth-SGD convergence analysis is insufficient in our setting: the equivalence transforms SGD into a \emph{sequence} of time-varying objectives $\{f_t\}_{t\ge 1}$, instead of a fixed static one. Moreover, controlling the behaviour of $\{\nabla f_t\}$ does not automatically yield control over $\nabla f$.
\end{remark}

Given that $f$ is smooth, we have $f_{t}(\cdot)$ is $L_{t}$-smooth with $ L_{t} := L + \gamma_{t}$. Therefore, leveraging Proposition~\ref{prop:DA-reformulation} and the \(L_t\)-smoothness of every surrogate \(f_t\), we obtain  
\begin{equation}\label{eq:DA-central-bound}
	\sum_{t=1}^{T}\eta_t \left(1-\frac{\eta_t L_t}{2}\right)  \mathbb{E}\left[\norm{\nabla f_t(\bx_{t})}^{2} \right]  + \sum_{t=1}^{T} \mathbb{E}\left[f_{t-1}(\bx_{t}) - f_t(\bx_{t}) \right] \\
	\leq  \Delta+ \sum_{t=1}^{T} \frac{\eta_{t}^{2} L_{t}}{2}~\E\left[\norm{\bxi_t}^2\right]
\end{equation}
where $\Delta := f(\bzero) - f^*$. Equation~\eqref{eq:DA-central-bound} generalizes the classical SGD descent inequality from the stationary case \(f_t\equiv f\) to an online setting in
which the learner faces a time-varying sequence of smooth losses; so far, the derivation remains agnostic to how these losses are generated by the SDA updates.  In our context, however, the surrogates possess the special property
\begin{equation*}
	f_{t}(\bx_{t+1}) - f_t(\bx_t) = f_t(\bx_{t+1})-f_{t-1}(\bx_t) + \frac{\gamma_{t-1}-\gamma_t}{2} \norm{\bx_t}^2,
\end{equation*}
which shows that up to a telescopic term, the term $\sum_{t=1}^{T} \mathbb{E}\left[f_{t-1}(\bx_{t}) - f_t(\bx_{t}) \right]$ in \eqref{eq:DA-central-bound} carries a positive component $\frac{\gamma_{t-1}-\gamma_t}{2} \mathbb{E}[\norm{\bx_t}^2 ]$ allowing us to relate the obtained bound on the gradients of $f_t$ to one on the gradients of objective function $f$. To achieve this, a particular selection of learning rates, detailed in the lemma below, is required.

\begin{lemma}\label{lem:tech_main}
    Consider the step sizes
    \[
        \eta_t = \frac{1}{L(1+\rho) (1+\rho + \alpha \sqrt{t})}~,
    \]
    where $\alpha = \min \left\lbrace \frac{\sigma}{L(1+\rho)}, 1 \right\rbrace$. We have for any $t \ge 2$, $\frac{\eta_t L_t}{2} \le 1$ and 
    \begin{align*}
           \frac{1}{2} \eta_t -\frac{1}{4} \eta_t^2 L_t - \frac{1}{2}\rho \eta_t^2L_t &\ge \frac{\eta_t}{8}\\
           \gamma_{t-1}-\gamma_t -\gamma_t^2\eta_t &\ge 0 ~.
    \end{align*}
\end{lemma}

\begin{theorem}\label{thm:exp}
    Suppose the $(\rho, \sigma)$-strong-growth condition holds, and that $f$ is $L$-smooth, then SDA iterates with step sizes
    \[
    \eta_t = \frac{1}{L(1+\rho) (1+\rho + \alpha \sqrt{t})},
    \] 
    where  $\alpha = \min \left\lbrace \frac{\sigma}{L(1+\rho)}, 1 \right\rbrace$, satisfy
    \begin{align*}
        &\frac{1}{T}\sum_{t=1}^T \mathbb{E}\left[ \norm{\nabla f(\bx_t)}^2\right] \\
		&\le \left(10 \Delta + 4 \left(\frac{\sigma^2}{L(1+\rho)}+ L(1+\rho)\right)\log(1+\alpha \sqrt{T}) \right) \frac{L(1+\rho)^2 + \sigma \sqrt{T}}{T}.
    \end{align*}
    where $\Delta = f(\mathbf{0})-f^*$.
\end{theorem}
\begin{proof}[Proof sketch]
	A detailed proof is presented in Section~\ref{sec:proof_sg} of the Appendix. We consider the equivalence argument introduced in Proposition~\ref{prop:DA-reformulation}.
	Because $f_t$ is $L_t:=L+\gamma_t$-smooth, the standard descent inequality gives
	\[
	f_t(\bx_{t+1})\le
	f_t(\bx_t)-\eta_t\!\left(1-\tfrac{\eta_tL_t}{2}\right)\!
	\bigl\langle\nabla f_t(\bx_t),\widehat{\nabla}f_t(\bx_t)\bigr\rangle
	+\tfrac{\eta_t^2L_t}{2}\bigl\|\widehat{\nabla}f_t(\bx_t)\bigr\|^2~.
	\]
	Adding and subtracting $f_{t-1}(\bx_t)$ inside the left-hand side makes the term
	$f_t(\bx_{t+1})-f_{t-1}(\bx_t)$ telescope once we sum over~$t$.  After taking expectations and using the noise bound
	$\E_{t-1}[\|\bxi_t\|^2]\le\rho\|\nabla f(\bx_t)\|^2+\sigma^2$, we arrive at
	\[
	\sum_{t=1}^T\E[A_t] \le f(\mathbf 0)-f^*+\tfrac{\sigma^2}{2}\sum_{t=1}^T\eta_t^2L_t,
	\]
	where
	\[
	A_t =\eta_t\!\Bigl(1-\tfrac{\eta_tL_t}{2}\Bigr)\!\|\nabla f_t(\bx_t)\|^2 +\tfrac{\gamma_{t-1}-\gamma_t}{2}\|\bx_t\|^2 -\tfrac{\rho}{2}\eta_t^2L_t\|\nabla f(\bx_t)\|^2~.
	\]
	Choosing the stepsizes
	$\eta_t = \frac{1}{L(1+\rho) (1+\rho + \alpha \sqrt{t})}$, 
	Lemma \ref{lem:tech_main} ensures that $A_t\ge\tfrac{\eta_t}{8}\|\nabla
	f(\bx_t)\|^2$.  Substituting this lower bound and evaluating the resulting sums
	yields
	\[
	\sum_{t=1}^T\E\bigl[\|\nabla f(\bx_t)\|^2\bigr] \le \frac{10}{\eta_T} \Delta + \frac{4 \sigma^2 \log(1+\alpha \sqrt{T})}{\eta_T L(1+\rho)\alpha^2} , \quad \Delta:=f(\mathbf 0)-f^*,
	\]
	which coincides with the statement of Theorem \ref{thm:exp}.
\end{proof}
Our bound scales as $\mathcal{O}\left(1/T + \sigma \log(T)/\sqrt{T}\right)$. When the noise level satisfies \(\sigma>0\), this matches the minimax-optimal rate for stochastic gradient descent established by \citet{GhadimiL12}, up to a \(\log T\) factor.  Note that the learning rates we used \emph{does not} require advance knowledge of the time horizon \(T\), whereas the optimal stepsizes in \citet{GhadimiL12} are tuned using this information.  In the low noise regime (\(\sigma=0\))—which, under the strong growth condition,
corresponds to gradient noise being on the same order as the true gradient—our analysis shows that the accelerated \(1/T\) rate is attainable,
mirroring the best known guarantees for standard SGD. 

\begin{remark}[Comparison with \citet{JelassiD20}]
	Notably, the guarantee in \citet{JelassiD20} (Theorem 3.1) does not apply for the standard Dual Averaging (DA) iteration, but for a variant so called {Modernized Dual Averaging} (MDA) algorithm, and also with a very restricted setting where the scaling parameters $(\lambda_t,\beta_t)$ are required to grow like $\sqrt{t+1}$. Although the MDA algorithm encompasses DA as a special case $\beta_t\equiv 1$ and $\lambda_t=\eta_t$, the specific choice of parameters in Theorem 3.1 (the only one for which they guarantee convergence) does not, as the gradients do not enter into the iterate with equal weights. 
	
	Even further, their guarantee is stated under the \emph{a-priori} assumption that the iterates remain in a fixed Euclidean ball, i.e.\ $\sup_{t\ge 1}\|\bx_t\|\le R$ for some known radius $R$.  Our result removes this technical condition: the bound in Theorem~\ref{thm:exp} holds \emph{without} any explicit control on the iterates and is therefore \emph{iterate-independent} and, as a result, proves that SDA converges at the optimal rates. Furthermore, even in settings where the iterates are bounded, the result of \citet{JelassiD20} fails to recover the fast $\mathcal{O}(1/T)$ rate attainable in low-noise regimes, whereas our analysis does.
\end{remark}

\section{High probability convergence of SDA}
\label{sec:high-prob}
The bound we have just established controls \emph{the average} performance over a hypothetical ensemble of infinitely many independent runs. Such guarantees in-expectation are of theoretical value, but they do not tell us what happens in typical machine learning applications, where practitioners only perform a few number of runs rather than a whole ensemble. Therefore, to mover closer to this practical setting, a stronger result is desired, to ensure that each run will, with probability at least $1 - \delta$, stay within the required error tolerance. To achieve a high-probability bound amounts to controlling the \emph{entire tail} of the error distribution, which apparently requires stronger assumptions on the stochastic noise distribution. Empirical evidences \citep{zhang2020adaptive,simsekli2019tail} has shown that the noise distribution for standard vision tasks can be well-approximated by a sub-Gaussian distribution. In this section, we adopt the strong-growth condition \ref{ass:strong-growth} with $\rho = 0$ and further add \ref{ass:subgaussian-noise} to our assumption set, and show in Theorem~\ref{thm:da-convergence} a high probability convergence guarantee of SDA.

We define for $t \geq 1$
\begin{align*}
    \hn_t(\bx) := \hn(\bx)+ \gamma_t \bx; \quad \Delta_{\tau}f_{t} := f_{t}(\bx_{\tau}) - f^{*}; \quad \bxi_{t} := \hn(\bx_{t}) - \nabla f(\bx_{t}).
\end{align*}
We let $\F_{t} := \sigma\left(\xi_{1}, \dots, \xi_{t-1}\right)$ denote the natural filtration. Note that $\bx_{t}$ is $\F_{t}$-measurable. The following lemma serves as a fundamental step of our analysis.
\begin{lemma} \label{lem:da-basic-inequality}
For $t \geq 1$, we have
\begin{align}
    C_{t} &:= \eta_{t} \left(1 - \frac{\eta_{t} L_{t}}{2}\right) \norm{\nabla f_{t}(\bx_{t})}^{2} + \Delta_{t+1}f_{t} - \Delta_{t}f_{t}\nonumber \\
    & \leq \eta_{t} \left(\eta_{t} L_{t} - 1\right) \inn{\nabla f_{t}(\bx_{t})}{\bxi_{t}} + \frac{\eta_{t}^{2} L_{t}}{2}\norm{\bxi_{t}}^{2}. \label{eq:da-basic-inequality}
\end{align}
\end{lemma}
We can concentrate the LHS of \eqref{eq:da-basic-inequality} using a similar argument from \cite{liu2023high}. For $w_{t} \geq 0$, let
\[
    Z_{t} := w_{t} C_{t} - v_{t} \norm{\nabla f_{t}(\bx_{t})}^{2}; \quad S_{t} := \sum_{i=t}^{T} Z_{i}, \quad \forall \, 1 \leq t \leq T
\]
where $v_{t} = 3 \sigma^{2} w_{t}^{2} \eta_{t}^{2} (\eta_{t} L_{t} - 1)^{2}$. Observe that the proof of Theorem 4.3 in the paper applies to any sequence $\{\ba_t\}$ that only includes noises up to $\bx_{t}$, rather than the specific choice of $\{\nabla f(\bx_{t})\}$. Using the same technique, we can prove the following key inequality.
\begin{theorem} \label{thm:da-moment-inequality} 
Suppose for all $1 \leq t \leq T$, $w_{t}$ and $\eta_{t}$ satisfy
\begin{equation} \label{eq:da-scale-condition}
    0 \leq w_{t} \eta_{t}^{2} L_{t} \leq \dfrac{1}{2 \sigma^{2}},
\end{equation}
then
\begin{equation}
    \E \left[\exp\left(S_{t}\right) \mid \F_{t}\right] \leq \exp\left(3 \sigma^{2} \sum_{i=t}^{T} \frac{w_{i} \eta_{i}^{2} L_{i}}{2}\right). \label{eq:da-key-inequality}
\end{equation}
\end{theorem}
Markov's inequality gives us the following guarantee in probability.
\begin{corollary} \label{cor:da-general-guarantee}
If condition \eqref{eq:da-scale-condition} holds for all $1 \leq t \leq T$, then with probability at least $1 - \delta$, we have
\begin{equation} \label{eq:da-general-bound}
    \begin{split}
        &\sum_{t=1}^{T} \left(\left(w_{t} \eta_{t} \left(1 - \frac{\eta_{t} L_{t}}{2}\right) - v_{t}\right) \norm{\nabla f_t(\bx_{t})}^{2} + w_{t}\left(\Delta_{t+1}f_{t} - \Delta_{t}f_{t}\right)\right) \\
        &\leq 3 \sigma^{2} \sum_{t=1}^{T} \frac{w_{t} \eta_{t}^{2} L_{t}}{2} + \log \frac{1}{\delta}.
    \end{split}
\end{equation}
\end{corollary}
Equipped with Lemmas \ref{lem:da-basic-inequality} and \ref{thm:da-moment-inequality}, we are ready to prove the central lemma for SDA by specifying
the choice of $w_{t}$ that satisfy the condition of Lemma \ref{thm:da-moment-inequality}.
\begin{lemma} \label{lem:da-central-bound} Suppose the learning rates $\{\eta_{t}\}$ is non-increasing with $\eta_t L_t \leq 1$, then for any $\delta > 0$, and define the auxiliary function $f_{0}(\bx) := f(\bx) + \gamma_0/2 \norm{\bx}^{2}$, where $\gamma_0$ is an arbitrary constant. Then, the following event holds with probability at least $1 - \delta$:
\begin{equation} \label{eq:da-central-bound}
    \begin{split}
        &\sum_{t=1}^{T} \eta_t \norm{\nabla f_t(\bx_{t})}^{2} + 2 \left(\Delta_{T+1} f_{T+1} - \Delta_{1} f_{1}\right) + \sum_{t=1}^{T} \left(f_{t-1}(\bx_{t}) - f_t(\bx_{t})\right) \\
        &\leq 3 \sigma^{2} \sum_{t=1}^{T} \eta_{t}^{2} L_{t} + 12 \sigma^{2} \eta_{1} \log \frac{1}{\delta}.
    \end{split}
\end{equation}
\end{lemma}

\begin{theorem} \label{thm:da-convergence}
Assume $f(\cdot)$ is $L$-smooth and satisfies \ref{eq:lower-boundedness}, \ref{eq:unbiased-estimator}, and \ref{eq:subgaussian-noise}. 
Setting $\eta_{t} = 1 / \left(L + \sigma \sqrt{t}\right)$, then with probability at least $1 - \delta$, the iterate sequence $\{x_{t}\}_{t \geq 1}$
output by SDA satisfies
\begin{equation} \label{eq:da-convergence}
    \frac{1}{T} \sum_{t=1}^{T} \norm{\nabla f(x_{t})}^{2} \leq \left(6 \Delta + 18 \left(L + \sigma\right) \log\left(1 + \frac{\sigma \sqrt{T}}{L}\right) + \frac{36 \sigma^{2}}{L + \sigma} \log\frac{1}{\delta}\right) \frac{\left(L + \sigma \sqrt{T}\right)}{T}.
\end{equation}      
\end{theorem}
Again, in the case where the time horizon $T$ is unknown to the algorithm, by choosing the step size $\eta$ in Theorem \ref{thm:da-convergence}, the bound is adaptive to noise, i.e, when $\sigma = 0$ we recover $\mathcal{O}(1 / T)$ convergence rate of the (deterministic) gradient descent algorithm.
\section{Adaptive Dual Averaging}
\label{sec:ada-DA}

In the previous section, we showed that it is possible to match the rate of SGD with stochastic SDA on non-convex smooth functions. However, the optimal learning rate in \eqref{eq:eta} depends on the unknown variance of the noise, $\sigma^2$. Here, we show that, as for SGD, we can design an adaptive version of SDA, using AdaGrad-norm stepsizes~\citep{StreeterM10}, given by
\begin{equation*}
    \eta_t \;=\; \frac{\eta}{\sqrt{\gamma + \sum_{i=1}^{t}\!\norm{\bg_i}^2}}\, .
\end{equation*}
With this choice we can bound in Theorem~\ref{thm:ada} the average squared gradient as  
\[
    \mathcal{O}\left( \sigma \sqrt{\frac{\E\left[\max_{t\le T} \norm{\bx_t-\bx^*}^4\right]}{T}}+ \E\left[\max_{t\le T} \norm{\bx_t-\bx^*}^4\right]\right).
\]
Hence, when the iterates remain bounded we recover the optimal rate. Similar bounds with a dependence on the maximal iterate norm, were presented in earlier works for SGD with adaptive step sizes \cite{DuchiHS11}.
\begin{theorem}\label{thm:ada}
    Assume $f(\cdot)$ is $L$-smooth and satisfies \ref{eq:lower-boundedness}, \ref{eq:unbiased-estimator}, and \ref{eq:strong-growth}. Consider SDA iterates using step sizes
    \[
        \eta_t = \frac{\eta}{\sqrt{\gamma + \sum_{i=1}^t \norm{\bg_i}^2}}~.
    \] 
    Let $\bx^*$ be a global minimizer and $B_T := \max_{t\le T} \ \left\{ \norm{\bx_t - \bx^*}\norm{\bx_{t}} + \frac{1}{2} \norm{\bx_t - \bx^*}^2 +\norm{\bx_{t}}^2\right\}~$. Then, we have
    \begin{align*}
        \sum_{t=1}^T \mathbb{E}\left[ \norm{\nabla f(\bx_t)}^2\right] 
        &\le 2\sqrt{2}L\sigma \left(\mathbb{E}\left[\left(\frac{B_T}{\eta}+2\eta\right)^2 \right]\right)^{1/2} \sqrt{T} \\
		&\quad + 2 \sqrt{2}L^2\left(\sqrt{2}+\rho\right) \mathbb{E}\left[\left(\frac{B_T}{\eta}+2\eta\right)^2 \right]
        + \frac{\gamma}{2}+\frac{2\left(f(\mathbf{0})-f^*\right)}{\eta}~.
    \end{align*}
\end{theorem}
\begin{proof}[Proof sketch]
	Full details are deferred to Appendix~\ref{sec:pr_ada}; here we record the main chain of inequalities in continuous prose.  
	Because SDA keeps the dual accumulator $\btheta_t=\sum_{i=1}^{t}\bg_i$, its recursion can be written as  
	\[
	\bx_{t+1}=-\eta_t\btheta_t =\bx_t-\eta_t\!\bigl(\bg_t+(1-\eta_{t-1}/\eta_t)\btheta_{t-1}\bigr) =\bx_t-\eta_t\bg'_t,
	 \]
	where  \( \bg'_t:=\bg_t+\bigl(1-\eta_{t-1}/\eta_t\bigr)\btheta_{t-1}\) and  \(\eta_t=\eta/\sqrt{\gamma+\sum_{i=1}^{t}\|\bg_i\|^{2}}=\eta/\sqrt{S_t}.\) The $L$-smoothness of $f$ yields the descent estimate  
	\[
		f(\bx_{t+1})-f(\bx_t) \le-\eta_t\langle\nabla f(\bx_t),\bg'_t\rangle +\tfrac{L}{2}\eta_t^{2}\|\bg'_t\|^{2}.
	\]
	Summing over $t$ and expanding $\bg'_t$ produces a telescoping term  $\Delta_1/\eta+\sum_{t}(1/\eta_t-1/\eta_{t-1})\Delta_t$
	plus the variance contribution  $\frac{L}{2}\!\sum_{t}\eta_t\|\bg'_t\|^{2}$. Gives
	\[
		\sum_{t=1}^T \nabla f(\bx_t)^{\top}\bg_t \le \frac{\Delta_1}{\eta} - \sum_{t=1}^T \!\left(\nabla f(\bx_t)^{\top} \bx_{t} - \Delta_t - L\norm{\bx_{t}}^2\right) \!\left(\frac{1}{\eta_t}-\frac{1}{\eta_{t-1}}\right) \! + L\eta \sum_{t=1}^T \frac{\norm{\bg_t}^2}{\sqrt{S_t}}.
	\]
	Since  \(\sum_{t}\|\bg_t\|^{2}/\sqrt{S_t}\le2\sqrt{S_T},\) the latter collapses to $2L\eta\sqrt{S_T}$.  
	Next, the strong-growth condition together with the smoothness bound  \( \Delta_t\le\frac{L}{2}\|\bx_t-\bx^*\|^{2}\) and the inequality  \(\|\nabla f(\bx_t)\|\le L\|\bx_t-\bx^*\|\)
	allow every term involving $\bx_t$ or $\bx_t-\bx^*$ to be controlled by the single radius $B_T$.  
	A direct rearrangement then shows
	\[
	\sum_{t=1}^{T}\langle\nabla f(\bx_t),\bg_t\rangle \le \frac{\Delta_1}{\eta} +2\Bigl(\tfrac{LB_T}{\eta}+2L\eta\Bigr)^{2} +\frac{1}{4}\sum_{t=1}^{T}\|\nabla f(\bx_t)\|^{2} +\frac{\gamma}{4} +\sqrt{2}\Bigl(\tfrac{LB_T}{\eta}+2L\eta\Bigr) \sqrt{\sum_{t=1}^{T}\|\bxi_t\|^{2}}.
	\]
	Finally, taking expectations and applying Cauchy-Schwarz to the noise term changes  \(\sqrt{\sum_{t}\|\bxi_t\|^{2}}\)
	into $\sqrt{T}\,\sigma$, completing the derivation of the bound claimed in the theorem. 
\end{proof}

\begin{remark}
Observe that the upper bound is minimized when $\gamma=0$. Indeed, we can safely use $\gamma=0$ by observing that we can avoid divisions by zero by not updating the iterates every time the stochastic gradient is zero.
\end{remark}

Although Theorem \ref{thm:ada} delivers the desired convergence rate whenever the iterates remain bounded, it still falls short of the \emph{iterate-independent} upper-bounds guarantees obtained in Section \ref{sec:strong-growth} for deterministic step sizes. The main difficulty stems from the SDA update, given by
$\bx_{t+1} - \bx_t =  - \eta_t \bg_t +\left(1-\eta_t/\eta_{t-1}\right) \bx_t$, 
which includes an extra term in $\bx_t$ absent from the SGD recursion.  
With deterministic step sizes, the ratio between factors $\eta_t$ and $\lvert \eta_t/\eta_{t-1}-1 \rvert$ is $\Theta(1/\sqrt{t})$, so this term is can be controlled; under the AdaGrad rule no such uniform decay holds.  
As a result, the update direction becomes a delayed, biased combination of past gradients, and in non-convex landscapes those past gradients could be uninformative for the current update.
Classical techniques such as the descent lemma therefore do not suffice, and developing tight, iterate-independent convergence guarantees for adaptive SDA remains an open research avenue.

\section{Conclusion, Limitations and Future Work}
\label{sec:conclusion}

We have revisited \emph{dual averaging} through the lens of smooth non-convex optimisation and provided the first \textit{iterate-level} guarantees that match the $\mathcal{O}(T^{-1/2})$ complexity enjoyed by carefully tuned SGD, without restrictive assumptions.
Concretely, we first proved sharp convergence bounds for deterministic step sizes under the strong-growth condition; next, we extended these guarantees to high-probability statements in the presence of sub-Gaussian noise; and finally, we established adaptive rates for AdaGrad-style learning rates that apply whenever the iterates remain bounded.

\paragraph{Limitations and Future Work.} Our analysis comes with a few caveats. First, the high probability guarantees rely on sub-Gaussian gradient noise. Yet several empirical studies \citep{chen2021heavy,nguyen2019first} reveal that stochastic-gradient noise can be heavy-tailed in certain deep-learning tasks, for instance BERT transformer model for natural language processing. It suggests that the sub-Gaussian setting may be overly optimistic, in which case, our rates may no longer hold. Second, the adaptive Ada-DA variant only converges provably when the iterates are bounded --- obtaining truly iterate-independent, adaptive bounds remains an open problem.

A promising next step is to derive \emph{iterate-independent} guarantees for DA equipped with fully adaptive step sizes, paralleling the bounds already known for SGD with AdaGrad. Achieving this will likely require new control of the cumulative gradient history; one avenue is to adapt the bias-correction arguments introduced for momentum-type first-order methods. In particular, the analyses developed in the case of Adam and related optimizers (see e.g., \citet{zhou2018convergence}, \citet{chen2018convergence}) face the same core challenge of handling updates that combine all past gradients with time-varying weights.

\newpage

\bibliographystyle{plainnat_nourl}
\bibliography{references}

\newpage

\appendix
\section{Appendix} 
\label{sec:appendix}

\subsection{Missing Proofs in Section \ref{sec:strong-growth}}
\label{sec:proof_sg}
\begin{reptheorem}[Restatement of Theorem~\ref{thm:exp}]
  Suppose the $(\rho, \sigma)$-strong-growth condition holds, and that $f$ is $L$-smooth, then SDA iterates with step sizes
  \[
  \eta_t = \frac{1}{L(1+\rho) (1+\rho + \alpha \sqrt{t})},
  \] 
  where  $\alpha = \min \left\lbrace \frac{\sigma}{L(1+\rho)}, 1 \right\rbrace$, satisfy
  \begin{align*}
  	&\frac{1}{T}\sum_{t=1}^T \mathbb{E}\left[ \norm{\nabla f(\bx_t)}^2\right] \\
  	&\le \left(10 \Delta + 4 \left(\frac{\sigma^2}{L(1+\rho)}+ L(1+\rho)\right)\log(1+\alpha \sqrt{T}) \right) \frac{L(1+\rho)^2 + \sigma \sqrt{T}}{T}.
  \end{align*}
  where $\Delta = f(\mathbf{0})-f^*$.
\end{reptheorem}
\begin{proof}
     We have
    \begin{align*}
        \bx_{t+1} &= -\eta_t (\btheta_{t-1}+\bg_t)\\
        &= -\eta_{t-1} \btheta_{t-1} \frac{\eta_{t}}{\eta_{t-1}} -\eta_t \bg_t\\
        &= \bx_t \left(1 - \left(1-\frac{\eta_{t}}{\eta_{t-1}}\right)\right) - \eta_t \bg_t\\
        &= \bx_t -\eta_t \left( \bg_t + \left(\frac{1}{\eta_{t-1}}-\frac{1}{\eta_t}\right) \bx_t\right)~.
    \end{align*}
    Let $\gamma_t := \frac{1}{\eta_{t}}-\frac{1}{\eta_{t-1}}$, and $f_t(\bx) = f(\bx)+ \frac{\gamma_t}{2}\norm{\bx}^2$, recall that 
    \[
    \nabla f_t(\bx) = \nabla f(\bx)+\gamma_t \bx~.
    \]
    Therefore, the iterates of SDA using the steps sizes $\eta_t$ are the same as the iterates of SGD with steps sizes $\eta_t$ on the sequence of functions $(f_t)_{t\ge 1}$ with stochastic gradients $\widehat{\nabla} f(\bx_t) := \bg_t + \gamma_t \bx_t$. 
    
    \noindent We consider $\eta_0 = \eta_1$, and take the convention $\gamma_{-1} =  0$. Define $L_t := L+\gamma_t$, since $f$ is $L$-smooth, $f_t$ is $L_t$ smooth, therefore
    \begin{align}
        f_t(\bx_{t+1}) -f_t(\bx_t) &\le \inn{\nabla f_t(\bx_t)}{\bx_{t+1}-\bx_{t}} + \frac{L_t}{2}\norm{\bx_{t+1}-\bx_t}^2\nonumber\\
        &= -\eta_t \inn{\nabla f_t(\bx_t)}{\widehat{\nabla}f_t(\bx_t)} + \frac{\eta_t^2L_t}{2}\norm{\widehat{\nabla}f(\bx_t)}^2\nonumber\\
        &= -\eta_t \norm{\nabla f_t(\bx_t)}^2 -\eta_t \inn{\nabla f_t(\bx_t)}{\xi_t} +\frac{\eta_t^2L_t}{2} \norm{\nabla f(\bx_t)}^2+\frac{\eta_t^2L_t}{2} \norm{\bxi_t}^2 \\
        &\qquad +\eta_t^2L_t\inn{\nabla f(\bx_t)}{\bxi_t}\nonumber\\
        &=-\eta_t \left(1-\frac{\eta_tL_t}{2} \right) \norm{\nabla f_t(\bx_t)}^2 -\eta_t (1-\eta_tL_t) \inn{\nabla f(\bx_t)}{\bxi_t} \\
        &\qquad+ \frac{\eta_t^2 L_t}{2} \norm{\bxi_t}^2~.\label{eq:b1}
    \end{align}
    Observe that
    \begin{align}
        f_t(\bx_{t+1})-f_t(\bx_t) &= f_t(\bx_{t+1}) -f_{t-1}(\bx_t)+f_{t-1}(\bx_t) - f_{t}(\bx_t)\nonumber\\
        &= f_t(\bx_{t+1})-f_{t-1}(\bx_t) + \frac{\gamma_{t-1}-\gamma_t}{2} \norm{\bx_t}^2\label{eq:b2}
    \end{align}
    Combining \eqref{eq:b1} and \eqref{eq:b2} and rearranging we obtain
    \begin{align*}
        &\eta_t\left(1-\frac{\eta_t L_t}{2}\right)\norm{\nabla f_t(\bx_t)}^2 + \frac{\gamma_{t-1}-\gamma_{t}}{2} \norm{\bx_t}^2 \\
        &\le f_{t-1}(\bx_t)-f_t(\bx_{t+1})  -\eta_t (1-\eta_tL_t)\inn{\nabla f(\bx_t)}{\bxi_t} + \frac{\eta_t^2L_t}{2} \norm{\bxi_t}^2~.
    \end{align*}
    Taking the expectation, summing over $t$ and using the bound on the noise, we obtain
    \begin{align*}
        &\sum_{t=1}^T \eta_t\left(1-\frac{\eta_t L_t}{2}\right)\E\left[\norm{\nabla f_t(\bx_t)}^2 \right] + \frac{\gamma_{t-1}-\gamma_{t}}{2} \E\left[\norm{\bx_t}^2 \right] \\
        &\le f(\mathbf{0})  - \E\left[ f_T(\bx_{T+1})\right]+ \sum_{t=3}^T \frac{\eta_t^2L_t}{2} \E\left[\E_{t-1}\left[\norm{\bxi_t}^2\right]\right] \nonumber\\
        &\le f(\mathbf{0}) - f^*+ \sum_{t=3}^T \frac{\eta_t^2L_t}{2} \left(\rho \E\left[\norm{\nabla f(\bx_{t})}^2 \right] +\sigma^2\right)~. 
    \end{align*}
    Therefore,
    \begin{align*}
        &\sum_{t=1}^T \eta_t\left(1-\frac{\eta_t L_t}{2}\right)\E\left[\norm{\nabla f_t(\bx_t)}^2 \right] +\sum_{t=1}^T \frac{\gamma_{t-1}-\gamma_{t}}{2} \E\left[\norm{\bx_t}^2 \right] -\sum_{t=1}^T \frac{\rho}{2} \eta_t^2 L_t~\E\left[\norm{\nabla f(\bx_t)}^2\right]\\
        &\le f(\mathbf{0})-f^* + \sum_{t=1}^{T} \sigma^2\frac{\eta_t^2L_t}{2}~.
    \end{align*}	
    Denote 
    \[
    A_t := \eta_t\left(1-\frac{\eta_t L_t}{2}\right)\norm{\nabla f_t(\bx_t)}^2 +\frac{\gamma_{t-1}-\gamma_{t}}{2} \norm{\bx_t}^2  -  \frac{\rho}{2} \eta_t^2 L_t~\norm{\nabla f(\bx_t)}^2~.
    \]
    Therefore, the last bound writes
    \begin{equation}\label{eq:b3}
        \sum_{t=1}^T \mathbb{E}\left[A_t\right] \le \E\left[f_1(\bx_2)-f^*\right] + \frac{\sigma^2}{2}\sum_{t=1}^{T} \eta_t^2L_t~.
    \end{equation}
    We consider stepsizes of the form
    \[
    \eta_t = \frac{1}{L(1+\rho)(1+\rho+\alpha\sqrt{t})}~,
    \]
    where $\alpha = \min \left\lbrace \frac{\sigma}{L(1+\rho)}, 1 \right\rbrace$.
     Recall that we have 
    \begin{equation*}
        \norm{\nabla f_t(\bx_t)}^2 = \norm{\nabla f(\bx_t) + \gamma_t \bx_t}^2 \ge \frac{1}{2} \norm{\nabla f(\bx_t)}^2 - \gamma_t^2 \norm{\bx_t}^2~.
    \end{equation*}
    Therefore, following Lemma~\ref{lem:tech}, which gives that $1-\frac{\eta_tL_t}{2} \ge 0$, we have for any $t \ge 2$
    \begin{align*}
        A_t &\ge \frac{1}{2} \eta_t\left(1-\frac{\eta_t L_t}{2}\right)\norm{\nabla f(\bx_t)}^2  - \frac{1}{2}\eta_t \left(1-\frac{\eta_t L_t}{2}\right) \gamma_t^2 \norm{\bx_t}^2 + \frac{\gamma_{t-1}-\gamma_{t}}{2} \norm{\bx_t}^2 \\
        &\qquad - \frac{\rho}{2} \eta_t^2 L_t~\norm{\nabla f(\bx_t)}^2\\
        &= \left( \frac{1}{2} \eta_t -\frac{1}{4} \eta_t^2 L_t - \frac{1}{2}\rho \eta_t^2L_t\right) \norm{\nabla f(\bx_t)}^2 + \left(\frac{\gamma_{t-1}-\gamma_t}{2} -\frac{1}{2}\gamma_t^2\eta_t + \frac{1}{4}\gamma_t^2\eta_t^2 L_t\right) \norm{\bx_t}^2\\
        &\ge \frac{\eta_t}{8}  \norm{\nabla f(\bx_t)}^2~.  
    \end{align*}
    The last line follows from the lower bounds of the factors of $\norm{\nabla f(\bx_t)}^2$ and $\norm{\bx_t}^2$ provided in Lemma~\ref{lem:tech}. We conclude using \eqref{eq:b3} then the second bound of Lemma~\ref{lem:tech} that 
    \begin{align*}
        \sum_{t=2}^T \frac{\eta_t}{8}\mathbb{E}\left[\norm{\nabla f(\bx_t)}^2\right] &\le f(\mathbf{0})-f^* + \frac{\sigma^2}{2} \sum_{t=1}^T \eta_t^2L_t\\
        &\le f(\mathbf{0})-f^* + \frac{6\sigma^2}{10L(1+\rho)}  \sum_{t=1}^T \frac{1}{(1+\rho+\alpha \sqrt{t})^2}\\
        &\le f(\mathbf{0})-f^* + \frac{3\sigma^2}{5L(1+\rho)} \left(1+ \frac{2}{\alpha^2} \log \left(\frac{1+\rho+\alpha \sqrt{T}}{1+\rho+\alpha} \right)\right)\\
        &\le f(\mathbf{0})-f^* + \frac{3 \sigma^2 }{5L(1+\rho)} \left(1+\frac{2}{\alpha^2} \log(1+\alpha \sqrt{T})\right)~.
    \end{align*}
    Therefore,
    \begin{align*}
        \sum_{t=1}^T \mathbb{E}\left[\norm{\nabla f(\bx_t)}^2\right] &\le \norm{\nabla f(\mathbf{0})}^2+\frac{8}{\eta_T} (f(\mathbf{0})-f^*) + \frac{5 \sigma^2}{\eta_T L(1+\rho)} \left(1+\frac{2}{\alpha^2} \log(1+\alpha \sqrt{T})\right)\nonumber\\
        &\le \left(10 (f(\mathbf{0})-f^*) + \frac{4 \sigma^2 \log(1+\alpha \sqrt{T})}{ L(1+\rho)\alpha^2} \right) \frac{1}{\eta_T}~,
    \end{align*}
    where we used the descent lemma in the last line ($\norm{\nabla f(\mathbf{0})}^2 \le 2L(f(\mathbf{0})-f^*)$)~.
    We have using $\frac{x^2}{\min\{x^2, 1\}} \le x^2+1$
    \begin{align*}
    	\frac{4 \sigma^2 \log(1+\alpha \sqrt{T})}{ L(1+\rho)\alpha^2} &=  4\frac{\sigma^2}{L^2(1+\rho)^2}\frac{  \log(1+\alpha \sqrt{T})}{ \alpha^2}L(1+\rho)\\
    	&\le 4L(1+\rho) \left(\frac{\sigma^2}{L^2(1+\rho)^2} +1\right) \log(1+\alpha\sqrt{T})~. 
    \end{align*}
    The conclusion follows from the expressions of $\alpha$ and $\eta_T$
\end{proof}
    
\subsection{Missing Proofs in Section \ref{sec:high-prob}}

\begin{proof}[Proof of Lemma \ref{lem:da-basic-inequality}]

We start from the one-step descent of $f_{t}(\bx_{t})$. Smoothness of $f_{t}(\cdot)$ implies that
\begin{align*}
    f_{t}(\bx_{t+1}) - f_{t}(\bx_{t}) &\leq \inn{\nabla f_{t}(\bx_{t})}{\bx_{t+1} - \bx_{t}} + \frac{L_{t}}{2} \norm{\bx_{t+1} - \bx_{t}}^{2} \\
    &= -\eta_{t} \inn{\nabla f_{t}(\bx_{t})}{\hn_t(\bx_{t})} + \frac{\eta_t^{2} L_{t}}{2} \norm{\hn_t(\bx_{t})}^{2} \\
    &= -\eta_{t} \inn{\nabla f_{t}(\bx_{t})}{\nabla f_{t}(\bx_{t}) + \bxi_{t}} + \frac{\eta_t^{2} L_{t}}{2} \norm{\nabla f_{t}(\bx_{t}) + \bxi_{t}}^{2} \\
    &= -\eta_{t} \left(1 - \frac{\eta_t L_{t}}{2}\right) \norm{\nabla f_{t}(\bx_{t})}^{2} + \eta_{t} \left(\eta_t L_{t} - 1\right)\inn{\nabla f_{t}(\bx_{t})}{\bxi_{t}} + \frac{\eta_t^{2} L_{t}}{2} \norm{\bxi_{t}}^{2}.
\end{align*}
We obtain the inequality (\ref{eq:da-basic-inequality}) by rearranging the terms.
\end{proof}

\begin{proof}[Proof of Theorem \ref{thm:da-moment-inequality}]
We prove by induction. The base case $t = T + 1$ trivially holds. Consider
$1 \leq t \leq T$, we have
\begin{align*}
    \E\left[\exp\left(S_{t}\right) \mid \F_{t}\right] & = \E\left[\E\left[\exp\left(Z_{t} + S_{t+1}\right) \mid \F_{t+1}\right] \mid \F_{t}\right] \\
    &= \E\left[\exp\left(Z_{t}\right) \E\left[\exp\left(S_{t+1}\right) \mid \F_{t+1}\right] \mid \F_{k}\right].
\end{align*}
From the induction hypothesis we have 
\[
    \E\left[\exp\left(S_{t+1}\right) \mid \F_{t+1}\right] \leq \exp\left(3\sigma^{2}\sum_{i=t+1}^{T}\frac{w_{i}\eta_{i}^{2}L_{i}}{2}\right),
\]
hence 
\[
    \E\left[\exp\left(S_{t}\right) \mid \F_{t}\right] \leq \exp\left(3\sigma^{2} \sum_{i=t+1}^{T} \frac{w_{i} \eta_{i}^{2} L_{i}}{2}\right) \E\left[\exp\left(Z_{t}\right) \mid \F_{t}\right].
\]
We have then
\begin{align*}
    &\E\left[\exp\left(Z_{t}\right) \mid \F_{t}\right] \\
    &= \E\left[\exp\left(w_{t} \left(\eta_{t} \left(1 - \frac{\eta_{t} L_{t}}{2}\right) \norm{\nabla f_{t}(\bx_{t})}^{2} + \Delta_{t+1} f_{t} - \Delta_{t} f_{t}\right) - v_{t} \norm{\nabla f(\bx_{T})}^{2}\right) \mid \F_{t}\right] \\
    &\leq \E\left[\exp\left(w_{t} \left(\eta_{t} (\eta_{t} L_{t} - 1) \inn{\nabla f_{t}(\bx_{t})}{\bxi_{t}} + \frac{\eta_{t}^{2} L_{t}}{2} \norm{\bxi_{t}}^{2}\right) - v_{t} \norm{\nabla f_{t}(\bx_{t})}^{2}\right) \mid \F_{t}\right] \\
    &= \exp\left(-v_{t} \norm{\nabla f_{t}(\bx_{t})}^{2}\right) \E\left[\exp\left(w_{t}\left(\eta_{t} (\eta_{t} L_{t}-1) \inn{\nabla f_{t}(\bx_{t})}{\bxi_{t}} + \frac{\eta_{t}^{2} L_{t}}{2} \norm{\bxi_{t}}^{2}\right)\right) \mid \F_{t}\right] \\
    &\leq \exp\left(-v_{t} \norm{\nabla f_{t}(\bx_{t})}^{2}\right) \exp\left(3 \sigma^{2} \left(w_{t}^{2} \eta_{t}^{2} (\eta_{t} L_{t}-1)^{2} \norm{\nabla f_{t}(\bx_{t})}^{2} + \frac{w_{t} \eta_{t}^{2} L_{t}}{2}\right)\right) \\
    &= \exp\left(3 \sigma^{2} \frac{w_{t} \eta_{t}^{2} L_{t}}{2}\right)
\end{align*}
where the second line is due to (\ref{eq:da-basic-inequality}) in Lemma \ref{lem:da-basic-inequality} and the second to last line is due to the helper Lemma 2.2 in \cite{liu2023high}. Therefore,
\[
    \E\left[\exp\left(S_{t}\right) \mid \F_{t}\right] \leq \exp\left(3 \sigma^{2} \sum_{i=t}^{T} \frac{w_{i} \eta_{i}^{2} L_{i}}{2}\right)
\]
which is what we want to show.
\end{proof}

\begin{proof}[Proof of Corollary \ref{cor:da-general-guarantee}]
In Lemma \ref{thm:da-moment-inequality}, let $t = 1$ we obtain 
\[
    \E\left[\exp\left(S_{1}\right)\right] \leq \exp\left(3 \sigma^{2} \sum_{t=1}^{T} \frac{w_{t} \eta_{t}^{2} L_{t}}{2}\right).
\]
Hence, by Markov's inequality, we have 
\[
    \Pr\left[S_{1} \geq \left(3 \sigma^{2} \sum_{t=1}^{T} \frac{w_{t} \eta_{t}^{2} L_{t}}{2}\right) + \log\frac{1}{\delta}\right] \leq \delta.
\]
In other words, with probability at least $1 - \delta$ (once condition \eqref{eq:da-scale-condition} is satisfied),
\begin{align*}
    &\sum_{t=1}^{T} \left(\left(w_{t} \eta_{t} \left(1 - \frac{\eta_{t} L_{t}}{2}\right) - v_{t}\right) \norm{\nabla f_t(\bx_{t})}^{2} + w_{t} \left(\Delta_{t+1} f_{t} - \Delta_{t} f_{t}\right)\right) \\
    &\leq 3 \sigma^{2} \sum_{t=1}^{T} \frac{w_{t} \eta_{t}^{2} L_{t}}{2} + \log\frac{1}{\delta}.
\end{align*}
\end{proof}

\begin{proof}[Proof of Lemma \ref{lem:da-central-bound}]
Consider $w_{t} = w = 1 / (6 \sigma^{2} \eta_{1})$. Clearly, our choice of $w_{t}$ satisfies condition \eqref{eq:da-scale-condition} since 
\[
    w_{t} \eta_{t}^{2} L_t = \frac{\eta_{t}}{\eta_{1}} \cdot \left(\eta_{t} L_{t}\right) \cdot \frac{1}{6 \sigma^{2}} \leq \frac{1}{2 \sigma^{2}},
\]
then it follows that
\begin{align*}
    \text{LHS of \eqref{eq:da-general-bound}} &\geq \underbrace{\sum_{t=1}^{T} \left(w \eta_{t} \left(1 - \frac{\eta_{t} L_{t}}{2}\right) - 3 \sigma^{2} w^{2} \eta_{t}^{2} (\eta_{t} L_{t} - 1)^{2}\right) \norm{\nabla f_{t}(\bx_{t})}^{2}}_{A} \\
    &\quad + \underbrace{\sum_{t=1}^{T} w \left(\Delta_{t+1} f_{t} - \Delta_{t} f_{t}\right)}_B,
\end{align*}
where
\begin{align*}
    A &= \sum_{t=1}^{T} w \eta_{t} \left(1 - \frac{\eta_{t} L_{t}}{2} - 3 \sigma^{2} w \eta_{t} \left(1 - \eta_{t} L_{t}\right)^{2}\right) \norm{\nabla f_t(\bx_{t})}^{2} \\
    &\geq \sum_{t=1}^{T} w \eta_{t} \left(1 - \frac{\eta_{t} L_{t}}{2} - 3 \sigma^{2} w \eta_{1} \left(1 - \eta_{t} L_{t}\right)^{2}\right) \norm{\nabla f_t(\bx_{t})}^{2} \\
    &= \sum_{t=1}^{T} w \eta_{t} \left(1 - \frac{\eta_{t} L_{t}}{2} - \frac{1}{2} \left(1 - \eta_{t} L_{t}\right)^{2}\right) \norm{\nabla f_t(\bx_{t})}^{2} \\
    &\geq \sum_{t=1}^{T} \frac{w \eta_{t}}{2} \norm{\nabla f_t(\bx_{t})}^{2} \\
    &= \frac{w}{2} \sum_{t=1}^{T} \eta_t \norm{\nabla f_t(\bx_{t})}^{2}.
\end{align*}
The second inequality is due to 
\[
    1 - \frac{\eta L_{t}}{2\sqrt{t}} - \frac{1}{2}\left(1 - \eta_{t} L_{t}\right)^{2} \geq \frac{1}{2}
\]
when $0 \leq \eta_{t} L_{t} \leq 1$. We also have
\begin{align*}
    B &= w \sum_{t=1}^{T} \left(f_{t}(\bx_{t+1}) - f_{t}(\bx_{t})\right) \\
    &= w \sum_{t=1}^{T} \left(f_{t}(\bx_{t+1}) - f_{t-1}(\bx_{t})
    + f_{t-1}(\bx_{t}) - f_t(\bx_{t})\right) \\
    &= w \sum_{t=1}^{T} \left(f_{t}(\bx_{t+1}) - f_{t-1}(\bx_{t})\right) + w \sum_{t=1}^{T} \left(f_{t-1}(\bx_{t}) - f_t(\bx_{t})\right) \\
    &= w \left(\Delta_{T+1} f_{T+1} - \Delta_{1} f_{1}\right) + w \sum_{t=1}^{T} \left(f_{t-1}(\bx_{t}) - f_t(\bx_{t})\right).
\end{align*}
Therefore, with probability at least $1 - \delta$, we have
\begin{align*}
    &\sum_{t=1}^{T} \eta_t \norm{\nabla f_t(\bx_{t})}^{2} + 2 \left(\Delta_{T+1} f_{T+1} - \Delta_{1} f_{1}\right) + \sum_{t=1}^{T} \left(f_{t-1}(\bx_{t}) - f_t(\bx_{t})\right) \\
    &\leq 3 \sigma^{2} \sum_{t=1}^{T} \eta_{t}^{2} L_{t} + \frac{2}{w} \log \frac{1}{\delta} \\
    &= 3 \sigma^{2} \sum_{t=1}^{T} \eta_{t}^{2} L_{t} + 12 \sigma^{2} \eta_{1} \log \frac{1}{\delta}.
\end{align*}
\end{proof}

\begin{proof}[Proof of Theorem \ref{thm:da-convergence}]
We first take care of the LHS of \eqref{eq:da-central-bound}. Note that by definition,
\[
    \nabla f_{t}(\bx_{t}) = \nabla f(\bx_{t}) + \gamma_{t} \bx_{t}.
\]
Invoking Lemma \ref{lem:norm-decomposition} with $u = \nabla f(\bx_{t})$, $v = \gamma_{t}\bx_{t}$ and $\lambda = 2 / 3$ gives
\begin{equation} \label{eq:specific-norm-decomposition}
    \norm{\nabla f_{t}(\bx_{t})}^{2} \geq \frac{1}{3} \norm{\nabla f(\bx_{t})}^{2} - \frac{1}{2} \gamma_{t}^{2} \norm{\bx_{t}}^{2}.
\end{equation}
Also recall that
\begin{align}
    f_{t-1}(\bx_{t}) - f_t(\bx_{t}) &= f(\bx_{t}) + \frac{\gamma_{t-1}}{2} \norm{\bx_{t}}^{2} - \left(f(\bx_{t}) + \frac{\gamma_{t}}{2} \norm{\bx_{t}}^{2}\right) \nonumber \\
    &= \frac{\gamma_{t-1} - \gamma_{t}}{2} \norm{\bx_{t}}^{2} \label{eq:ft-diff}.
\end{align}
Obviously our choice of ${\eta_t}$ is non-increasing and
\begin{align*}
    \eta_{t} L_{t} &= \eta_{t} \left(L + \frac{1}{\eta_{t}} - \frac{1}{\eta_{t - 1}} \right) \\
    &= \eta_{t} L + 1 - \frac{\eta_{t}}{\eta_{t - 1}} \\
    &= \frac{L}{L + \sigma \sqrt{t}} - \frac{L + \sigma \sqrt{t - 1}}{L + \sigma \sqrt{t}} + 1 \\
    &\leq 1.
\end{align*}
Now, plugging the above couple of inequalities \eqref{eq:specific-norm-decomposition} and \eqref{eq:ft-diff} in, we obtain
\begin{align*}
    \text{LHS of \eqref{eq:da-central-bound}} &\geq \sum_{t=1}^{T} \eta_t \left(\frac{1}{3} \norm{\nabla f(\bx_{t})}^{2} - \frac{1}{2} \gamma_{t}^{2} \norm{\bx_{t}}^{2}\right) + 2 \left(\Delta_{T+1} f_{T+1} - \Delta_{1} f_{1}\right) \\
    &\quad + \sum_{t=1}^{T} \frac{\gamma_{t-1} - \gamma_{t}}{2} \norm{\bx_{t}}^{2} \\
    &= \frac{1}{3} \sum_{t=1}^{T} \eta_t \norm{\nabla f(\bx_{t})}^{2} + 2 \left(\Delta_{T+1} f_{T+1} - \Delta_{1} f_{1}\right) + \frac{1}{2} \sum_{t=1}^{T} \left(\gamma_{t-1} - \gamma_{t} - \gamma_{t}^{2} \eta_{t}\right) \norm{\bx_{t}}^{2} \\
    &\stackrel{(*)}{\geq} \frac{1}{3} \sum_{t=1}^{T} \eta_t \norm{\nabla f(\bx_{t})}^{2} + 2 \left(\Delta_{T+1} f_{T+1} - \Delta_{1} f_{1}\right) \\
    &\geq \frac{1}{3} \sum_{t=1}^{T} \eta_{T} \norm{\nabla f(\bx_{t})}^{2} + 2 \left(\Delta_{T+1} f_{T+1} - 2 \Delta_{1} f_{1}\right) \\
    &\geq \frac{\eta_{T}}{3} \sum_{t=1}^{T} \norm{\nabla f(\bx_{t})}^{2} - 2 \Delta_{1} f_{1}
\end{align*}
where $(*)$ comes from Lemma \ref{lem:offseting}, and the last inequality follows from the fact that $f_{T+1}(\bx_{T+1}) \geq f^{*}$. Besides, by choosing $\eta_{t} = 1 / (L + \sigma t)$,
\begin{align*}
    \text{RHS of \eqref{eq:da-central-bound}} &= 3 \sigma^{2} \sum_{t=1}^{T} \left(\frac{1}{L + \sigma \sqrt{t}}\right)^{2} \left(L + \sigma \sqrt{t} - \left(L + \sigma \sqrt{t - 1}\right)\right) + 12 \sigma^{2} \eta_{1} \log \frac{1}{\delta} \\
    &= 3 \sigma^{2} \sum_{t=1}^{T} \left(\frac{1}{L + \sigma \sqrt{t}}\right)^{2} \left(L + \sigma \left(\sqrt{t} - \sqrt{t-1}\right)\right) + 12 \sigma^{2} \eta_{1} \log \frac{1}{\delta} \\
    &\leq 3 \sigma^{2} \sum_{t=1}^{T} \left(\frac{1}{L + \sigma \sqrt{t}}\right)^{2} \left(L + \sigma\right) + 12 \sigma^{2} \eta_{1} \log \frac{1}{\delta} \\
    &\leq 3 \sigma^{2} \cdot \frac{2}{\sigma^{2}} \left(\log\left(1 + \frac{\sigma \sqrt{T}}{L}\right) + \frac{L}{L + \sigma \sqrt{T}} - 1\right) \left(L + \sigma\right) + 12 \sigma^{2} \eta_{1} \log \frac{1}{\delta} \\
    &\leq 6 \left(L + \sigma\right) \log\left(1 + \frac{\sigma \sqrt{T}}{L}\right) + 12 \sigma^{2} \eta_{1} \log \frac{1}{\delta}.
\end{align*}
Now, dividing both sides by $\eta_{T} / 3$ yields
\begin{align*}
    &\sum_{t=1}^{T} \norm{\nabla f(\bx_{t})}^{2} \\
    &\leq \frac{3}{\eta_{T}} \left(2 \left(f(\bzero) - f^*\right) + 6 \left(L + \sigma\right) \log\left(1 + \frac{\sigma \sqrt{T}}{L}\right) + 12 \sigma^{2} \eta_{1} \log \frac{1}{\delta}\right) \\
    &= \left(6 \left(f(\bzero) - f^*\right) + 18 \left(L + \sigma\right) \log\left(1 + \frac{\sigma \sqrt{T}}{L}\right) + \frac{36 \sigma^{2}}{L + \sigma} \log \frac{1}{\delta}\right) \left(L + \sigma \sqrt{T}\right).
\end{align*}
Factoring out an order of $T$ gives the desired result.
\end{proof}

\subsection{Proof of Theorem~\ref{thm:ada}}\label{sec:pr_ada}
\begin{reptheorem}[Restatement of Theorem~\ref{thm:ada}]
    Let $f$ be $L$-smooth, suppose the gradient noise $\bxi_t =  \bg_t - \nabla f(\bx_t)$ satisfies the $(\rho, \sigma)$ strong growth condition. Consider SDA iterates using step sizes
    \[
        \eta_t = \frac{\eta}{\sqrt{\gamma + \sum_{i=1}^t \norm{\bg_i}^2}}~.
    \] 
    Let $\bx^*$ be a global minimizer and
    \[
        B_T := \max_{t\le T} \ \left\{ \norm{\bx_t - \bx^*}\norm{\bx_{t}} + \frac{1}{2} \norm{\bx_t - \bx^*}^2 +\norm{\bx_{t}}^2\right\}~.
    \]
    Then, we have
    \begin{align*}
        \sum_{t=1}^T \mathbb{E}\left[ \norm{\nabla f(\bx_t)}^2\right] 
        &\le 2\sqrt{2}L\sigma \left(\mathbb{E}\left[\left(\frac{B_T}{\eta}+2\eta\right)^2 \right]\right)^{1/2} \sqrt{T} \\
        &\quad + 2 \sqrt{2}L^2\left(\sqrt{2}+\rho\right) \mathbb{E}\left[\left(\frac{B_T}{\eta}+2\eta\right)^2 \right] + \frac{\gamma}{2}+2\frac{f(\mathbf{0})-f^*}{\eta}~.
    \end{align*}
\end{reptheorem}
\begin{proof}
We consider the update rule given by
\[
\bx_{t+1}
=-\eta_t \sum_{i=1}^{t} \bg_i = -\eta_t \btheta_t,
\]
where $\btheta_t = \sum_{i=1}^{t} \bg_i$, and $\eta_t$ is a learning rate. Therefore
\begin{align*}
    \bx_{t+1}-\bx_t &= -\eta_{t} \btheta_t +\eta_{t-1} \btheta_{t-1}
    = (\eta_{t-1}-\eta_{t}) \btheta_{t-1} - \eta_{t} \bg_t
    = -\eta_{t} \left(\bg_t + \left(1 - \frac{\eta_{t-1}}{\eta_{t}}\right)\btheta_{t-1}\right)~.
\end{align*}
Denote $\bg'_t := \bg_t + \left(1 - \frac{\eta_{t-1}}{\eta_{t}}\right)\btheta_{t-1}$, the updates are $\bx_{t+1} = \bx_t - \eta_{t} \bg'_t$. 
We consider the learning rates given by
\begin{equation}\label{eq:eta_def}
    \eta_t = \frac{\eta}{\sqrt{\gamma + \sum_{i=1}^t \norm{\bg_i}^2} },
\end{equation}
We define $S_t := \gamma+\sum_{i=1}^{t} \norm{\bg_i}^2$. We note by $\bxi_t$ the noise
\[
\bxi_t :=  \bg_t - \nabla f(\bx_t)~.
\]
\noindent The updates of SDA are $\bx_{t+1} = \bx_t - \eta_t \bg'_t$. We have using smoothness
\begin{align*}
    f(\bx_{t+1})-f(\bx_t) 
    &\le \nabla f(\bx_t)^{\top} (\bx_{t+1}-\bx_{t}) + \frac{1}{2}L \norm{\bx_{t+1}-\bx_t}^2
    = -\eta_t \nabla f(\bx_t)^{\top} \bg'_{t} +\frac{L}{2} \eta_t^2 \norm{\bg'_t}^2~. 
\end{align*}
Let $\Delta_t := f(\bx_t)-f^*$. We are using the last bound
\[
    \nabla f(\bx_t)^{\top} \bg'_t \le \frac{1}{\eta_t}\left(\Delta_t - \Delta_{t+1} \right) + \frac{L}{2} \eta_t \norm{\bg'_t}^2~.
\]
Summing over $t$ and taking $\eta_0 = \eta$, we have
\begin{align*}
    \sum_{t=1}^{T} \nabla f(\bx_t)^{\top} \bg'_t &\le \sum_{t=1}^T \frac{1}{\eta_t} \left(\Delta_t - \Delta_{t+1}\right) + \frac{L}{2} \sum_{t=1}^T \eta_t \norm{\bg'_t}^2\\
    &\le \frac{\Delta_1}{\eta}-\frac{\Delta_{T+1}}{\eta_T} + \sum_{t=1}^{T} \left(\frac{1}{\eta_t} - \frac{1}{\eta_{t-1}}\right) \Delta_t + \frac{L}{2} \sum_{t=1}^{T} \eta_t \norm{\bg'_t}^2\\
    &\le \frac{\Delta_1}{\eta} + \sum_{t=1}^{T} \left(\frac{1}{\eta_t} - \frac{1}{\eta_{t-1}}\right) \Delta_t + \frac{L}{2} \sum_{t=1}^{T} \eta_t \norm{\bg'_t}^2~.
\end{align*}
Therefore, we obtain
\begin{align}
    &\sum_{t=1}^T \nabla f(\bx_t)^{\top}\bg_t \nonumber\\
    &\le \frac{\Delta_1}{\eta} + \sum_{t=1}^T \left( \frac{\eta_{t-1}}{\eta_t}-1\right)\nabla f(\bx_t)^{\top} \btheta_{t-1} + \sum_{t=1}^T  \left(\frac{1}{\eta_t} - \frac{1}{\eta_{t-1}}\right) \Delta_t + \frac{L}{2} \sum_{t=1}^{T} \eta_t \norm{\bg'_t}^2 \nonumber\\
    &\le \frac{\Delta_1}{\eta} - \sum_{t=1}^T \left(\frac{1}{\eta_t}-\frac{1}{\eta_{t-1}}\right) \nabla f(\bx_t)^{\top} \bx_{t} + \sum_{t=1}^T \left(\frac{1}{\eta_t}-\frac{1}{\eta_{t-1}}\right) \Delta_t\nonumber\\
    &\qquad  + L\sum_{t=1}^T \eta_t \norm{\bg_t}^2+ L \sum_{t=1}^T \eta_t \left(\frac{\eta_{t-1}}{\eta_t}-1\right)^2 \norm{\btheta_{t-1}}^2\nonumber\\
    &=  \frac{\Delta_1}{\eta} - \sum_{t=1}^T \left(\nabla f(\bx_t)^{\top} \bx_{t} - \Delta_t\right)\left(\frac{1}{\eta_{t}}-\frac{1}{\eta_{t-1}}\right) + L \sum_{t=1}^T \eta_t \norm{\bg_t}^2\nonumber\\
    &\qquad + L \sum_{t=1}^T \eta_t \left(\frac{1}{\eta_t}-\frac{1}{\eta_{t-1}}\right)^2 \norm{\bx_{t}}^2\nonumber\\
    &= \frac{\Delta_1}{\eta} - \sum_{t=1}^T \left(\nabla f(\bx_t)^{\top} \bx_{t} - \Delta_t\right)\left(\frac{1}{\eta_{t}}-\frac{1}{\eta_{t-1}}\right) + L \eta\sum_{t=1}^T \frac{\norm{\bg_t}^2}{\sqrt{S_t}}\nonumber\\
    &\qquad  + L \sum_{t=1}^T \left(1-\frac{\eta_t}{\eta_{t-1}}\right) \left(\frac{1}{\eta_{t}}-\frac{1}{\eta_{t-1}}\right) \norm{\bx_{t}}^2\nonumber\\
    &\le \frac{\Delta_1}{\eta} - \sum_{t=1}^T \left(\nabla f(\bx_t)^{\top} \bx_{t} - \Delta_t\right)\left(\frac{1}{\eta_t}-\frac{1}{\eta_{t-1}}\right) + L\eta \sum_{t=1}^T \frac{\norm{\bg_t}^2}{\sqrt{S_t}}\nonumber \\
    &\qquad +  L \sum_{t=1}^T  \left(\frac{1}{\eta_t}-\frac{1}{\eta_{t-1}}\right) \norm{\bx_{t}}^2\nonumber\\
    &\le \frac{\Delta_1}{\eta} - \sum_{t=1}^T \!\left(\nabla f(\bx_t)^{\top} \bx_{t} - \Delta_t - L\norm{\bx_{t}}^2\right) \!\left(\frac{1}{\eta_t}-\frac{1}{\eta_{t-1}}\right) \! + L\eta \sum_{t=1}^T \frac{\norm{\bg_t}^2}{\sqrt{S_t}}. \!\label{eq:boundn1}
\end{align}
We also have
\[
    \sum_{t=1}^T \frac{\norm{\bg_t}^2}{\sqrt{S_t}} 
    \le 2\sum_{t=1}^T \left(\sqrt{S_t}-\sqrt{S_{t-1}}\right)
    \le 2\sqrt{S_T}~.
\]
Let $\bx^*$ denote a stationary point, given that $f$ is $L$-smooth, we have
\begin{equation}
    \Delta_t = f(\bx_t)-f^*
    \le \nabla f(\bx^*)^{\top} (\bx_t-\bx^*) + \frac{L}{2} \norm{\bx_t-\bx^*}^2
    = \frac{L}{2} \norm{\bx_t-\bx^*}^2~.\label{eq:boundn2}
\end{equation}
Moreover,
\[
    \norm{\nabla f(\bx_t)} 
    = \norm{\nabla f(\bx_t)-\nabla f(\bx^*)}
    \le L \norm{\bx_t - \bx^*}~. 
\]
Using the last two bounds we have
\begin{align}
    \lvert \nabla f(\bx_t)^{\top} \bx_{t} - \Delta_t - L \norm{\bx_{t}}^2 \rvert 
    &\le \lvert \nabla f(\bx_t)^{\top} \bx_{t} \rvert + \lvert \Delta_t \rvert + L\norm{\bx_{t}}^2\nonumber\\
    &\le L \norm{\bx_t - \bx^*}\norm{\bx_{t}} + \frac{L}{2} \norm{\bx_t - \bx^*}^2 +L\norm{\bx_{t}}^2~.  \label{eq:boundn3}
\end{align}
Using bounds \eqref{eq:boundn2} and \eqref{eq:boundn3} in \eqref{eq:boundn1} and the definition of $B_T$, we have
\begin{align*}
    \sum_{t=1}^T \nabla f(\bx_t)^{\top}\bg_t &\le \frac{\Delta_1}{\eta} + LB_T \sum_{t=1}^T \left(\frac{1}{\eta_t} - \frac{1}{\eta_{t-1}}\right)+ 2L\eta\sqrt{S_T}\\
    &\le  \frac{\Delta_1}{\eta} + \frac{1}{\eta}LB_T \sqrt{S_T}+ 2L\eta\sqrt{S_T}\\
    &\le \frac{\Delta_1}{\eta} +\sqrt{2} \frac{L}{\eta}B_T \sqrt{\gamma + \sum_{t=1}^T \norm{\nabla f(\bx_t)}^2 + \sum_{t=1}^T \norm{\bxi_t}^2}\\
    &\qquad+ 2\sqrt{2}L\eta\sqrt{\gamma + \sum_{t=1}^T \norm{\nabla f(\bx_t)}^2 + \sum_{t=1}^T \norm{\bxi_t}^2}\\
    &\le \frac{\Delta_1}{\eta} +\sqrt{2} \left(\frac{L}{\eta}B_T+ 2L\eta\right) \sqrt{\gamma +\sum_{t=1}^T \norm{\nabla f(\bx_t)}^2}\\
    &\qquad+ \sqrt{2}\left(\frac{L}{\eta}B_T+ 2L\eta\right) \sqrt{\sum_{t=1}^T \norm{\bxi_t}^2}~.
\end{align*}
Taking the expectation both sides, we obtain
\begin{align}
    \sum_{t=1}^T \mathbb{E}\left[\norm{\nabla f(\bx_t)}^2 \right] &= \sum_{t=1}^T \mathbb{E}\left[ \nabla f(\bx_{t})^{\top} \bg_t\right]\nonumber\\
    &\le \frac{\Delta_1}{\eta} + \sqrt{2}~\mathbb{E}\left[\left(\frac{L}{\eta}B_T+ 2L\eta\right) \sqrt{\gamma+\sum_{t=1}^T \norm{\nabla f(\bx_t)}^2} \right] \nonumber\\
    & \qquad + \sqrt{2}~\mathbb{E}\left[ \left(\frac{L}{\eta}B_T+ 2L\eta\right) \sqrt{\sum_{t=1}^T \norm{\bxi_t}^2}\right]~. \label{eq:boundn6}
\end{align}
Recall that using the fact that $ax-\frac{1}{4}x^2 \le a^2$ for all $x\in \R$ we have
\begin{equation}\label{eq:boundn4}
    \sqrt{2}\left(\frac{L}{\eta}B_T+ 2L\eta\right) \sqrt{\gamma+\sum_{t=1}^T \norm{\nabla f(\bx_t)}^2} - \frac{1}{4} \sum_{t=1}^T \norm{\nabla f(\bx_t)}^2-\frac{\gamma}{4} \le  2\left( \frac{L}{\eta}B_T + 2L\eta\right)^2~.
\end{equation}
Moreover, we have
\begin{align*}
    \mathbb{E}&\left[ \left(\frac{L}{\eta}B_T+ 2L\eta \right)\sqrt{\sum_{t=1}^T \norm{\bxi_t}^2} \right] \\
    &\le \left(\mathbb{E}\left[\left(\frac{L}{\eta}B_T+ 2L\eta \right)^2 \right]\right)^{1/2} \left( \mathbb{E}\left[\sum_{t=1}^T \norm{\bxi_t}^2 \right]\right)^{1/2}\\
    &\le \left(\mathbb{E}\left[\left(\frac{L}{\eta}B_T+ 2L\eta \right)^2 \right]\right)^{1/2} \left( \mathbb{E}\left[\sum_{t=1}^T \left(\rho \norm{\nabla f(\bx_t)}^2 +\sigma^2\right) \right]\right)^{1/2}~.
\end{align*}
To ease the notation, let $\bar{B}_T :=  \left(\mathbb{E}\left[\left(\frac{1}{\eta}B_T+ 2\eta \right)^2 \right]\right)^{1/2}$. So, we have
\begin{align*}
    \mathbb{E}\left[ \left(\frac{L}{\eta}B_T+ 2L\eta \right)\sqrt{\sum_{t=1}^T \norm{\bxi_t}^2} \right] &\le L\bar{B}_T \sqrt{\sigma^2 T + \rho\sum_{t=1}^T \mathbb{E}\left[\norm{\nabla f(\bx_t)}^2\right] }\\
    &\le L\sigma \bar{B}_T \sqrt{T} + L\bar{B}_T \sqrt{\rho} \sqrt{\sum_{t=1}^T \E\left[\norm{\nabla f(\bx_t)}^2\right]}~. 
\end{align*}
Therefore, we have 
\begin{align}
    \mathbb{E}&\left[ \left(\frac{L}{\eta}B_T+ 2L\eta \right)\sqrt{\sum_{t=1}^T \norm{\bxi_t}^2} \right] - \frac{1}{4} \sum_{t=1}^T \mathbb{E}\left[\norm{\nabla f(\bx_t)}^2\right] \nonumber\\
    &\le L\sigma \bar{B}_T \sqrt{T} + L\bar{B}_T \sqrt{\rho} \sqrt{\sum_{t=1}^T \E\left[\norm{\nabla f(\bx_t)}^2\right]}
    - \frac{1}{4} \sum_{t=1}^T \mathbb{E}\left[\norm{\nabla f(\bx_t)}^2\right]\nonumber\\
    &\le L\sigma \bar{B}_T \sqrt{T} + L^2\bar{B}_T^2 \rho~. \label{eq:boundn5}
\end{align}
Using \eqref{eq:boundn4} and \eqref{eq:boundn5} in \eqref{eq:boundn6}, we have
\begin{align*}
    \frac{1}{2} \sum_{t=1}^T \mathbb{E}\left[\norm{\nabla f(\bx_t)}^2\right] &\le \frac{\gamma}{4}+\frac{\Delta_1}{\eta}+2\mathbb{E}\left[\left(\frac{L}{\eta}B_T+2L\eta\right)^2 \right] + \sqrt{2}L\sigma \bar{B}_T \sqrt{T} + \sqrt{2}L^2\rho \bar{B}_T^2\\
    &= \frac{\gamma}{4}+ \frac{\Delta_1}{\eta}+\sqrt{2}L\sigma \bar{B}_T \sqrt{T} + \sqrt{2}L^2\left(\sqrt{2}+\rho\right) \bar{B}_T^2~. \qedhere
\end{align*}	    
\end{proof}

\subsection{Additional Helper Lemmas}
\begin{lemma} \label{lem:norm-decomposition}
Let $u, v \in \R^d$ and let $\lambda > 0$.  Then
\begin{equation} \label{eq:norm-decomposition}
    \norm{u + v}^{2} \geq (1 - \lambda) \norm{u}^{2} + (1 - \lambda^{-1}) \norm{v}^{2}.
\end{equation}
\end{lemma}

\begin{proof}
Young's inequality says that for any $\lambda > 0$,
\[
    2 \inn{u}{v} \geq -\lambda^{-1}\norm{u}^{2} - \lambda \norm{v}^{2}.
\]
Insert this bound into the expansion
\[
    \norm{u + v}^{2} = \norm{u}^{2} + 2 \inn{u}{v} + \norm{v}^{2} ,
\]
yielding
\[
    \norm{u + v}^{2} \geq \norm{u}^2 - \lambda^{-1} \norm{u}^2 - \lambda  \norm{v}^2 +  \norm{v}^2 = 
    (1 - \lambda) \norm{u}^2 + (1 - \lambda^{-1}) \norm{v}^2~. \qedhere
\]
\end{proof}

\begin{lemma} \label{lem:offseting}
Let 
\[
    \eta_{t} = \frac{\eta}{1 + \alpha \sqrt{t}}, \quad \forall t \geq 1,
\]
with $\eta > 0$ and $\alpha \geq 0$. Then, we have
\begin{equation} \label{eq:offseting}
    \gamma_{t-1} - \gamma_{t} - \gamma_{t}^{2} \eta_{t} \geq 0, \quad \forall t \geq 0.
\end{equation}
\end{lemma}
\begin{proof}
Given the choice of $\eta_{t}$, we have access to a closed form expression for $\gamma_{t}$ in $t$. Let's find it.
\[
    \gamma_{t} = \frac{1}{\eta_{t}} - \frac{1}{\eta_{t-1}} = \frac{\alpha}{\eta} \left(\sqrt{t} - \sqrt{t - 1}\right).
\]
We will first examine the base cases. For $t = 1$, we need to show that
\[
    \gamma_{0} - \gamma_{1} - \gamma_{1}^{2} \eta_{1} = \gamma_{0} - \frac{\alpha}{\eta} - \frac{\alpha^{2}}{(1 + \alpha) \eta} \geq 0.
\]
Since we have control over $\gamma_{0}$, we can always set it to be sufficiently large. For example, setting 
\[
    \gamma_{0} = \frac{\alpha}{\eta} + \frac{\alpha^{2}}{(1 + \alpha) \eta} + 1
\]
will do the trick. For $t \geq 2$, 
\[
    \text{LHS of \eqref{eq:offseting}} = \alpha \underbrace{\left(2 \sqrt{t - 1} - \sqrt{t} - \sqrt{t - 2}\right)}_{A_{t}} + \alpha^{2} \sqrt{t} \underbrace{\left(4 \sqrt{(t-1)} - \sqrt{(t-2)} - 3 \sqrt{t} + \frac{1}{\sqrt{t}}\right)}_{B_{t}}.
\]
Here we intentionally omitted the boring algebra. Notice $A_{t}$ is non-negative for concavity of $\sqrt{x}$:
\begin{gather*}
    \frac{\sqrt{t}}{2} + \frac{\sqrt{t - 2}}{2}  \geq \sqrt{\frac{t}{2} + \frac{t - 2}{2}} = \sqrt{t - 1} \\
    \implies A_{t} = \frac{\sqrt{t}}{2} + \frac{\sqrt{t - 2}}{2} - \sqrt{t - 1} \geq 0.
\end{gather*}
To see that $B_{t}$ is non-negative, let
\[
    \psi(\tau) := 4 \sqrt{\tau - 1} - \sqrt{\tau - 2} - 3 \sqrt{\tau} + \frac{1}{\sqrt{\tau}}, \quad \forall \tau \in (2, \infty).
\]
Now, 
\begin{gather*}
    \psi(2) = 4 - 3 \sqrt{2} + \frac{1}{\sqrt{2}} \approx 0.464 > 0, \\
    \psi(\infty) = 0.
\end{gather*}
To show that $B_{t}$ is non-positive, it suffices to show that $\psi'(\tau)$ is non-negative for $\tau \in (2, \infty)$. Differentiating $\psi(\tau)$ gives 
\begin{align*}
    \psi'(\tau) &= \frac{2}{\sqrt{\tau - 1}} - \frac{1}{2 \sqrt{\tau - 2}} - \frac{3}{2 \sqrt{\tau}} - \frac{1}{2 \tau^{3 / 2}} \leq 0.
\end{align*}
Thus, $\psi$ decreases monotonically from the positive value $\psi(2)$ down to 0, so $\psi(\tau) \geq 0$ for all $\tau \geq 2$. In particular, $B_t=\psi(t) \geq 0$ for every integer $t \geq 2$.
\end{proof}

\begin{lemma}\label{lem:tech}
	Consider the step sizes
	\[
	\eta_t = \frac{1}{L(1+\rho) (1+\rho+\alpha \sqrt{t})}~,
	\]
	where $\alpha \in (0,1]$, let $\gamma_t = \frac{1}{\eta_t}-\frac{1}{\eta_{t-1}}$, and $L_t = L+\gamma_t$. We have for any $t \ge 2$
	\begin{align}
		\frac{\eta_t L_t}{2} &\le 1\label{eq:c1}\\
		\eta_t^2 L_t &\le \frac{6}{5L(1+\rho) (1+\rho+\alpha \sqrt{t})^2} \label{eq:c4}\\
		\frac{1}{2} \eta_t -\frac{1}{4} \eta_t^2 L_t - \frac{1}{2}\rho \eta_t^2L_t &\ge \frac{\eta_t}{8}\label{eq:c2}\\
		\frac{\gamma_{t-1}-\gamma_t}{2} -\frac{1}{2}\gamma_t^2\eta_t + \frac{1}{4}\gamma_t^2\eta_t^2 L_t &\ge 0 ~.\label{eq:c3}
	\end{align}
\end{lemma}
\begin{proof}
	Fix $t \ge 3$. Befor proving the results of the lemma, we use the following bounds on $\gamma_t$ and $\gamma_{t-1}-\gamma_t$. We have
	\begin{align}
		\gamma_{t} &= \frac{1}{\eta_t} - \frac{1}{\eta_{t-1}}\nonumber\\
		&= L(1+\rho) \left(1+\rho + \alpha \sqrt{t}\right)- L(1+\rho) \left(1+\rho + \alpha \sqrt{t-1}\right)\nonumber\\
		&= L(1+\rho) \alpha\left(\sqrt{t}-\sqrt{t-1}\right)\nonumber\\
		&= \frac{L(1+\rho)\alpha}{\sqrt{t}+\sqrt{t-1}} \le \frac{L(1+\rho)\alpha}{ 2\sqrt{t-1}}~.\label{eq:int1}
	\end{align}
	We also have
	\begin{align}
		\gamma_t^2 \eta_t &\le \frac{L^2(1+\rho)^2\alpha^2}{4(t-1)} \cdot \frac{1}{L(1+\rho) (1+\rho+\alpha \sqrt{t})}\nonumber\\
		&\le \frac{L(1+\rho) \alpha^2}{4(t-1)(1+\rho+\alpha \sqrt{t})} \nonumber\\
		&\le  \frac{L(1+\rho) \alpha}{4(t-1) \sqrt{t}}~.\label{eq:inter3}
	\end{align}	
	Moreover we have using \eqref{eq:int1}
	\begin{align*}
		c(\gamma_{t-1}-\gamma_t) &= \frac{L(1+\rho)\alpha}{(\sqrt{t-1}+\sqrt{t-2})}-\frac{L(1+\rho)\alpha}{(\sqrt{t}+\sqrt{t-1})}\\
		&= L(1+\rho)\alpha \left(\frac{1}{\sqrt{t-1}+\sqrt{t-2}}-\frac{1}{\sqrt{t}+\sqrt{t-1}}\right).
	\end{align*}
	Recall that 
	\begin{align*}
		\frac{1}{\sqrt{t-1}+\sqrt{t-2}}-\frac{1}{\sqrt{t}+\sqrt{t-1}} &= \frac{\sqrt{t}-\sqrt{t-2}}{(\sqrt{t-1}+\sqrt{t-2})(\sqrt{t}+\sqrt{t-1})}\\
		&= \frac{2}{(\sqrt{t}+\sqrt{t-2})(\sqrt{t-1}+\sqrt{t-2})(\sqrt{t}+\sqrt{t-1})}\\
		&\ge \frac{1}{4\sqrt{t}(t-1)}~.
	\end{align*}
	We conclude that 
	\begin{equation}\label{eq:inter2}
		\gamma_{t-1}-\gamma_t \ge \frac{L(1+\rho)\alpha}{4\sqrt{t}(t-1)}~.
	\end{equation}
	
	\noindent \textbf{Proof of bound \eqref{eq:c1}:} 
	We have
	\begin{align*}
		\eta_t L_t &\le \frac{1}{L(1+\rho) (1+\rho+\alpha \sqrt{t})} \left(L+\frac{L(1+\rho)\alpha}{\sqrt{t}}\right)\\
		&\le \frac{1}{1+\rho+\alpha \sqrt{t}} \left(\frac{1}{1+\rho}+ \frac{\alpha}{\sqrt{t}}\right)\\
		&= \frac{1/(1+\rho)}{1+\rho+\alpha \sqrt{t}} + \frac{\alpha }{\sqrt{t}(1+\rho+\alpha \sqrt{t})}\\
		&\le \frac{1}{(1+\rho)^2}+ \frac{\alpha}{\sqrt{t}(1+\rho)+ \alpha t}\\
		&\le \frac{1}{(1+\rho)^2}+ \frac{1}{\sqrt{t}(1+\rho)+ t}\le 2~,
	\end{align*}
	where we used $\alpha \in [0,1]$ and the fact that $x \to \frac{x}{\sqrt{t}(1+\rho)+xt}$ is increasing for positive numbers.
	
	\noindent \textbf{Proof of bound \eqref{eq:c4}:}
	We have
	\begin{align*}
		\eta_t^2 L_t &= \frac{1}{L^2(1+\rho)^2(1+\rho+\alpha \sqrt{t})^2} \left(L+\frac{L(1+\rho)\alpha}{4(t-1)\sqrt{t}}\right)\\
		&\le \frac{(6+\rho)}{5L(1+\rho)^2 (1+\rho+\alpha \sqrt{t})^2}\\
		&\le  \frac{6}{5L(1+\rho) (1+\rho+\alpha \sqrt{t})^2}
	\end{align*}

	\noindent \textbf{Proof of bound \eqref{eq:c2}:}
	We have
	\begin{align*}
		\eta_t-\frac{1}{2}\eta_t^2 L_t -\rho \eta_t^2L_t &= \eta_t \left(1- \left(\frac{1}{2} +\rho \right) \eta_tL_t\right)\\
		&\ge \eta_t \left(1- \frac{1+2\rho}{2} \left(\frac{1}{(1+\rho)^2}+\frac{1}{\sqrt{t}(1+\rho)+t}\right) \right)\\
		&\ge \eta_t \left(1- \frac{(1+2\rho)}{2(1+\rho)^2} -\frac{1+2\rho}{2\sqrt{2}(1+\rho)+4} \right)~.
	\end{align*}
	To conclude, we use that the maximal value of $h(\rho) = \frac{(1+2\rho)}{2(1+\rho)^2} +\frac{1+2\rho}{2\sqrt{2}(1+\rho)+4}$ over positive numbers in upper bounded by $3/4$.

	\noindent \textbf{Proof of bound \eqref{eq:c3}:}
	We have used the bound on $\gamma_{t-1}-\gamma_t$ given by \eqref{eq:inter2}, with the bound on $\gamma_t^2\eta_t$ given in \eqref{eq:inter3}.
	\begin{align*}
		\frac{\gamma_{t-1}-\gamma_t}{2} -\frac{1}{2}\gamma_t^2\eta_t + \frac{1}{4}\gamma_t^2\eta_t^2 L_t &\ge \frac{L(1+\rho)\alpha}{4\sqrt{t}(t-1)} -\frac{1}{2}\gamma_t^2 \eta_t \left(1-\frac{1}{2}\eta_t L_t\right)\\
		&\ge \frac{L(1+\rho)\alpha}{4\sqrt{t}(t-1)} -\frac{1}{2} \frac{L(1+\rho)\alpha}{4(t-1)\sqrt{t}} \left(1-\frac{1}{2} \frac{L}{L(1+\rho)} \right)\ge 0~,
	\end{align*}
	where we used in the last line $L_t \ge L$ and $\eta_t \le \frac{1}{L(1+\rho)}$ to have $\eta_tL_t \ge \frac{1}{1+\rho}$.

\end{proof}

\end{document}